\documentclass[reqno]{amsart}

\usepackage{amssymb, amsmath, amsthm, amscd, mathrsfs}

\usepackage{verbatim}
\usepackage{a4wide}
\usepackage{bbm}

\usepackage{enumerate}

\allowdisplaybreaks

\usepackage{hyperref}

\theoremstyle{plain}
\newtheorem{theorem}{Theorem}[section]
\theoremstyle{remark}
\newtheorem{remark}[theorem]{Remark}

\theoremstyle{plain}

\newtheorem{lemma}[theorem]{Lemma}
\newtheorem{proposition}[theorem]{Proposition}
\newtheorem{definition}[theorem]{Definition}

\numberwithin{equation}{section}

\def\N{{\mathbb N}}

\def\R{{\mathbb R}}

\def\H{{\mathbb H}}

\newcommand{\E}{{\mathbb E}}

\renewcommand{\b}{\beta}

\newcommand{\e}{\varepsilon}

\newcommand{\ph}{\varphi}

\renewcommand{\L}{\mathbb L}

\newcommand{\wt}{\tilde}

\newcommand{\one}{\mathbbm{1}}
\newcommand{\embed}{\hookrightarrow}

\def\typeout#1{\message{^^J}\message{#1}\message{^^J}}
\newif\ifSRCOK \SRCOKtrue
\newcount\PAGETOP
\newcount\LASTLINE
\global\PAGETOP=1
\global\LASTLINE=-1
\def\EJECT{\SRC\eject}
\def\WinEdt#1{\typeout{:#1}}
\gdef\MainFile{\jobname.tex}
\gdef\CurrentInput{\MainFile}
\newcount\INPSP
\global\INPSP=0
\def\SRC{\ifSRCOK%
  \ifnum\inputlineno>\LASTLINE%
    \ifnum\LASTLINE<0%
      \global\PAGETOP=\inputlineno%
    \fi%
    \global\LASTLINE=\inputlineno%
    \ifnum\INPSP=0%
      \ifnum\inputlineno>\PAGETOP%
        
      \fi%
    \else%
      
    \fi%
  \fi%
\fi}
\def\PUSH#1{%
\SRC%
\ifnum\INPSP=0 \global\let\INPSTACKA=\CurrentInput \else%
\ifnum\INPSP=1 \global\let\INPSTACKB=\CurrentInput \else%
\ifnum\INPSP=2 \global\let\INPSTACKC=\CurrentInput \else%
\ifnum\INPSP=3 \global\let\INPSTACKD=\CurrentInput \else%
\ifnum\INPSP=4 \global\let\INPSTACKE=\CurrentInput \else%
\ifnum\INPSP=5 \global\let\INPSTACKF=\CurrentInput \else%
               \global\let\INPSTACKX=\CurrentInput \fi\fi\fi\fi\fi\fi%
\gdef\CurrentInput{#1}%
\WinEdt{<+ \CurrentInput}%
\global\LASTLINE=0%
\ifSRCOK\fi%
\global\advance\INPSP by 1}
\def\POP{%
\ifnum\INPSP>0 \global\advance\INPSP by -1  \fi%
\ifnum\INPSP=0 \global\let\CurrentInput=\INPSTACKA \else%
\ifnum\INPSP=1 \global\let\CurrentInput=\INPSTACKB \else%
\ifnum\INPSP=2 \global\let\CurrentInput=\INPSTACKC \else%
\ifnum\INPSP=3 \global\let\CurrentInput=\INPSTACKD \else%
\ifnum\INPSP=4 \global\let\CurrentInput=\INPSTACKE \else%
\ifnum\INPSP=5 \global\let\CurrentInput=\INPSTACKF \else%
               \global\let\CurrentInput=\INPSTACKX \fi\fi\fi\fi\fi\fi%
\WinEdt{<-}%
\global\LASTLINE=\inputlineno%
\global\advance\LASTLINE by -1%
\SRC}
\def\INPUT#1{\relax}

\let\originalxxxeverypar\everypar
\newtoks\everypar
\originalxxxeverypar{\the\everypar\expandafter\SRC}
\everymath\expandafter{\the\everymath\expandafter\SRC}
\output\expandafter{\expandafter\SRCOKfalse\the\output}

\newif\ifSRCOK \SRCOKtrue
\newcount\PAGETOP
\newcount\LASTLINE
\global\PAGETOP=1
\global\LASTLINE=-1
\gdef\MainFile{\jobname.tex}
\gdef\CurrentInput{\MainFile}
\newcount\INPSP
\global\INPSP=0
\def\EJECT{\SRC\eject}
\def\WinEdt#1{\typeout{:#1}}
\def\SRC{\ifSRCOK%
  \ifnum\inputlineno>\LASTLINE%
    \ifnum\LASTLINE<0%
      \global\PAGETOP=\inputlineno%
    \fi%
    \global\LASTLINE=\inputlineno%
    \ifnum\INPSP=0%
      \ifnum\inputlineno>\PAGETOP%
      \fi%
    \else%
    \fi%
  \fi%
\fi}
\def\PUSH#1{%
\SRC%
\ifnum\INPSP=0 \global\let\INPSTACKA=\CurrentInput \else%
\ifnum\INPSP=1 \global\let\INPSTACKB=\CurrentInput \else%
\ifnum\INPSP=2 \global\let\INPSTACKC=\CurrentInput \else%
\ifnum\INPSP=3 \global\let\INPSTACKD=\CurrentInput \else%
\ifnum\INPSP=4 \global\let\INPSTACKE=\CurrentInput \else%
\ifnum\INPSP=5 \global\let\INPSTACKF=\CurrentInput \else%
               \global\let\INPSTACKX=\CurrentInput \fi\fi\fi\fi\fi\fi%
\gdef\CurrentInput{#1}%
\WinEdt{<+ \CurrentInput}%
\global\LASTLINE=0%
\ifSRCOK\fi%
\global\advance\INPSP by 1}
\def\POP{%
\ifnum\INPSP>0 \global\advance\INPSP by -1  \fi%
\ifnum\INPSP=0 \global\let\CurrentInput=\INPSTACKA \else%
\ifnum\INPSP=1 \global\let\CurrentInput=\INPSTACKB \else%
\ifnum\INPSP=2 \global\let\CurrentInput=\INPSTACKC \else%
\ifnum\INPSP=3 \global\let\CurrentInput=\INPSTACKD \else%
\ifnum\INPSP=4 \global\let\CurrentInput=\INPSTACKE \else%
\ifnum\INPSP=5 \global\let\CurrentInput=\INPSTACKF \else%
               \global\let\CurrentInput=\INPSTACKX \fi\fi\fi\fi\fi\fi%
\WinEdt{<-}%
\global\LASTLINE=\inputlineno%
\global\advance\LASTLINE by -1%
\SRC}
\def\INPUT#1{\relax}
\let\OldINCLUDE=\include
\def\include#1{
\EJECT%
\PUSH{#1.tex}%
\OldINCLUDE{#1}%
\POP}

\let\originalxxxeverypar\everypar
\newtoks\everypar
\originalxxxeverypar{\the\everypar\expandafter\SRC}
\everymath\expandafter{\the\everymath\expandafter\SRC}
\let\zzzxxxbibliography=\bibliography
\def\bibliography#1{\PUSH{\jobname.bbl}\zzzxxxbibliography{#1}\POP}
\output\expandafter{\expandafter\SRCOKfalse\the\output}

\begin{document}

\author{Lahcen Maniar}
\address{L. maniar, Cadi Ayyad University, Faculty of Sciences Semlalia, 2390, Marrakesh, Morocco}
\email{maniar@ucam.ac.ma}

\author{Martin Meyries}
\address{M. Meyries, Institut f\"ur Mathematik, Martin-Luther-Universit\"at Halle-Wittenberg, 
06099 Halle (Saale), Germany}
\email{martin.meyries@mathematik.uni-halle.de}

\author{Roland Schnaubelt}
\address{R. Schnaubelt, Department of Mathematics, Karlsruhe Institute of Technology (KIT), 
76128 Karlsruhe, Germany} 
\email{schnaubelt@kit.edu}

\title[Null controllability for dynamic boundary conditions]{Null controllability for parabolic 
equations with dynamic boundary conditions of reactive-diffusive type}


\thanks{We thank the Deutsche Forschungsgemeinschaft which supported this research within the
grants ME 3848/1-1 and SCHN 570/4-1. M.M. thanks L.M. for a very pleasant stay in Marrakesh, where parts of this work originated.}


\begin{abstract}
We prove null controllability for linear and semilinear heat equations with dynamic boundary 
conditions of surface diffusion type. The results are based on a new Carleman estimate for this type 
of boundary conditions.\end{abstract}

\maketitle


\section{Introduction}
In this paper we establish the null controllability of linear and 
semilinear parabolic equations with dynamic boundary conditions of surface diffusion type.
The prototype of such problems is
\begin{equation}\label{eq:intro} \left\{
\begin{aligned}
\partial_{t}y -\Delta y & = \one_\omega v(t,x)   & \qquad\quad &\text{in } (0,T)\times \Omega,\\
\partial_{t}y_\Gamma  -\Delta_\Gamma y_\Gamma +\partial_\nu y  & =  0  & &\text{on }(0,T)\times \Gamma, \\
(y,y_\Gamma)|_{t=0} & = (y_0, y_{0,\Gamma})  & &\text{in } \Omega\times \Gamma.
\end{aligned} \right.
\end{equation}
For all given $T>0$, $\omega \Subset \Omega$ and
initial data $y_0\in L^2(\Omega)$ in the bulk and $y_{0,\Gamma}\in L^2(\Gamma)$ on the boundary,
we want to find a control  $v\in L^2((0,T)\times\omega)$ such that the solution satisfies
\[y(T,\cdot) =0 \qquad \text{in \ } \overline{\Omega}.\]
Here $\Omega\subset \R^N$ is a bounded domain with smooth 
boundary $\Gamma = \partial\Omega$, $N\ge 2$, and the control region $\omega$ is an 
\emph{arbitrary} nonempty open subset which is strictly contained in $\Omega$ (i.e., $\overline{\omega}
\subset \Omega$). Further, $y_\Gamma = y|_\Gamma$ denotes the trace of a function $y:\Omega\to \R$, $\nu$ is the outer unit normal field, $\partial_\nu$ is the normal 
derivative at $\Gamma$, and $\Delta_\Gamma$ designates the Laplace-Beltrami operator on $\Gamma$.

The term $\partial_t y_\Gamma-\Delta_\Gamma y_\Gamma$ models the tangential diffusive flux on the boundary which is 
coupled to the diffusion equation in the bulk by the normal derivative $\partial_\nu y = (\nu\cdot \nabla y)|_{\Gamma}$. One may view \eqref{eq:intro} as a coupled system of dynamic equations for $y$ and $y_\Gamma$, with side condition $y|_{\Gamma} = y_\Gamma$. Sometimes this type of boundary 
conditions is called of Wentzell type. Dynamic surface and interface processes have attracted a lot of 
attention in recent years in the mathematical and applied literature,  see \cite{BPS, DPZ, EMR, FGGR, Gal, Glitzky, GM, Gold, HKR, MM, SW, VV}. In particular, for the mathematical theory of surface diffusion boundary conditions we refer to \cite{DPZ, FGGR, Gal, MM, SW, VV}. In Section \ref{sec:IBVP} we complete the existing $L^2$-based solution theory for \eqref{eq:intro} and its generalizations as needed in the context of null controllability. Here we obtain existence, uniqueness and regularity of strong, mild and distributional solutions.

We state main result ensuring the null controllability of \eqref{eq:intro}, see 
Theorem~\ref{thm:control-linear}. We emphasize that the initial data $y_0$ and $y_{0,\Gamma}$ on $\Omega$ 
and $\Gamma$ need not be related.
\begin{theorem} \label{thm:intro} 
For  each $T > 0$,  each nonempty open set $\omega \Subset \Omega$ and  all initial data
$y_0\in L^2(\Omega)$ and $y_{0,\Gamma}\in L^2(\Gamma)$ there is a control $v\in L^2((0,T)\times \omega)$ 
such that the unique strong solution $y$ of \eqref{eq:intro} satisfies $y(T,\cdot) =0$ on $\overline{\Omega}$.
\end{theorem}

Null controllability results of this type are known for Dirichlet and for  inhomogeneous or nonlinear Neumann boundary conditions 
(also called Robin or Fourier boundary conditions), see e.g.\ \cite{CIK,DFGZ, FGGP1, FGGP2, I95} and the 
survey article \cite{FCG}. Our Theorem~\ref{thm:control-linear} actually treats  a more general version 
of the control problem \eqref{eq:intro} with nonautonomous potential terms and inhomogeneities in the bulk
 and on the boundary. We also show null controllability for a semilinear variant of \eqref{eq:intro} 
including globally Lipschitz nonlinearities both in the bulk and on the boundary, see 
Theorem~\ref{thm:control-semilinear}. In view of the results in \cite{DFGZ, FGGP1, FGGP2, FCG}, we expect 
that one can extend the results to nonlinearities with slightly superlinear growth and containing gradient 
terms, using more involved regularity theory for \eqref{eq:intro}. For Dirichlet or Neumann boundary 
conditions one can also show  null controllability for controls acting at a part of the boundary by means of 
an auxiliary domain control problem on an enlarged spatial domain, see \cite{CIK} or \cite{DFGZ}. 
However, this trick does not seem to work for dynamical boundary conditions.

We are not aware of results on null controllability for parabolic problems with dynamical boundary 
conditions.  Optimal control problems for such equations were treated in \cite{HKR}.
 Approximate boundary controllability of a one-dimensional heat equation with dynamical boundary conditions
was studied  in \cite{KN} by completely different methods.

Theorem~\ref{thm:intro} relies on a Carleman estimate for the inhomogeneous dual problem corresponding to 
\eqref{eq:intro}, which is proved in Lemma~\ref{lem:carleman}. Roughly speaking, this estimate bounds a
weighted $L^2$-norm of the solution $\ph$ to the dual problem by a weighted $L^2$-norm 
of the inhomogeneities and of the restriction $\ph|_\omega$. These weights 
tend to zero exponentially as $t\to 0$ and $t\to T$. The proof of our Carleman estimate follows the known
strategy of the Dirichlet case, see \cite{FCG}, but the dynamic boundary condition leads to various new 
boundary terms. Some of these enter in the final estimate, a few cancel, and others can be controlled using 
the smoothing effect of the  surface diffusion in  \eqref{eq:intro}.  For this reason we cannot treat the 
corresponding equations without  surface diffusion, see Remark \ref{rem:delta=0} for more details.

With the Carleman estimate at hand, standard arguments yield an observability estimate 
 for the solution $\ph$ of the dual homogeneous backward problem \eqref{eq:4}--\eqref{eq:6},
see Proposition~\ref{lem:observa}. A simple transformation then yields the analogous
inequality
\[ \|y(T,\cdot)\|_{L^2(\Omega)}^2 +  \|y_\Gamma(T,\cdot)\|_{L^2(\Gamma)}^2   
\leq C \int_0^T \int_{\omega} |y|^2 \,dx\,dt,\]
for the solution of \eqref{eq:intro} with $v=0$. One calls this property the final state observability
since the observation on $\omega$ controls the state at the end time. By duality, the observability estimate for $\ph$ then yields the null  controllability  of \eqref{eq:intro}
as stated in Theorem~\ref{thm:intro}, see Theorem~\ref{thm:control-linear}. We refer to \cite{TW}
for a discussion of various controllability and observability concepts.

 The proof of the null controllability of the semilinear 
equation relies on Schauder's fixed point theorem. To set up this fixed point argument, one has
to construct a control with minimal weighted $L^2$-norm for the inhomogenous linear system involving nonautonomous 
potential terms. This optimization problem is solved in Proposition~\ref{lem:min}. In its proof we adopt
the methods of Imanuvilov's seminal paper \cite{I95} to the case of dynamic boundary conditions.

This paper is organized as follows. In Section \ref{sec:IBVP} we introduce the functional analytic setting
 and  prove basic results for (versions of) \eqref{eq:intro} concerning well-posedness in the framework of 
 strong, mild and distributional solutions. The Carleman estimate is proved in 
Section~\ref{sec:carleman} and the null controllability results are obtained in Section \ref{sec:null}.

\section{The initial-boundary value problem}\label{sec:IBVP}
Let $T > 0$ and a bounded domain $\Omega \subset \R^N$, $N\geq 2$, with smooth boundary $\Gamma = 
\partial \Omega$ and outer unit normal field $\nu$ on $\Gamma$ be given. We write
$$\Omega_T = (0,T)\times \Omega,  \qquad \omega_T = (0,T)\times \omega ,
\qquad \text{and}\qquad \Gamma_T = (0,T)\times \Gamma,$$
where $\omega\Subset \Omega$ is open. In this section we establish wellposedness and regularity properties of solutions of the 
inhomogeneous linear system
\begin{alignat}{2}
\partial_t y -d\Delta y +a(t,x)y& = f(t,x) &  \qquad\quad &\text{in }\Omega_T,\label{eq:4a} \\
\partial_t y -\delta \Delta_\Gamma y +d\partial_\nu y + b(t,x) y & = g(t,x) &  
        &\text{on }\Gamma_T, \label{eq:5a}\\
y(0,\cdot)  & = y_0 & &\text{in } \overline{\Omega}, \label{eq:6a}
\end{alignat}
for given coefficients $d,\delta > 0$, $a\in L^\infty(\Omega_T)$ and $b\in L^\infty(\Gamma_T)$. Here and in the rest of the paper, the initial condition \eqref{eq:6a} is short for $(y,y_\Gamma)|_{t=0} = (y_0,y_{0,\Gamma})$ in $\Omega\times \Gamma$, and we neglect the traces in \eqref{eq:5a}.

\subsection{Function spaces} 
The Lebesgue measure on $\Omega$ and the surface measure on $\Gamma$ are denoted by $dx$ and $d S$, 
respectively. We consider the real Hilbert spaces (and tacitly their complexifications if necessary)
\begin{align*}
\L^2 &:= L^2(\Omega)\times L^2(\Gamma), \qquad 
\H^k := \{(y,y_\Gamma)\in H^k(\Omega) \times H^k(\Gamma)\,:\,y|_\Gamma=y_\Gamma\}\quad\text{for \ }k\in\N,\\
\E_1(t_0,t_1) &:= H^1(t_0,t_1; \L^2) \cap L^2(t_0,t_1; \H^2) \quad \text{for } t_1>t_0 \text{ in } \R, 
  \qquad \E_1:=\E_1(0,T).
\end{align*}
The scalar product on $\L^2$ is given by
$$\langle (y,y_\Gamma), (z,z_\Gamma)\rangle_{\L^2} 
  = \langle y,z\rangle_{L^2(\Omega)} + \langle y_\Gamma, z_\Gamma\rangle_{L^2(\Gamma)}.$$ 
Further, $H^k(\Omega)$ are the usual $L^2$-based Sobolev spaces over $\Omega$. The spaces $H^k(\Gamma)$ are 
defined via local coordinates, see e.g. Definition 3.6.1 in \cite{Triebel}. At a few points we will also 
need the fractional order spaces $H^s(\Omega)$ and $H^s(\Gamma)$ with noninteger  $s \geq 0$. For our 
purposes it suffices to define them as interpolation spaces
$$H^s(\Omega) = (L^2(\Omega), H^2(\Omega))_{s/2,2}, 
\qquad H^s(\Gamma) = (L^2(\Gamma), H^2(\Gamma))_{s/2,2}, \qquad s\in [0,2],$$
where $(\cdot,\cdot)_{s/2,2}$ denotes the real interpolation functor, see Chapter~1 and Theorem~4.3.1/2 of
\cite{Triebel} or Chapter~1 of \cite{Lunardi-Int}. As a consequence, we obtain the interpolation inequalities
\begin{equation}\label{eq:interpol}
 \|y\|_{H^s(\Omega)} \leq C \|y\|_{L^2(\Omega)}^{1-s/2}\,\|y\|_{H^2(\Omega)}^{s/2}, \qquad 
\|y_\Gamma\|_{H^s(\Gamma)} \leq C \|y_\Gamma\|_{L^2(\Gamma)}^{1-s/2}\,\|y_\Gamma\|_{H^2(\Gamma)}^{s/2}, 
\qquad s\in [0,2],
\end{equation}
see Theorem~1.3.3 of \cite{Triebel}. 

For every $s > \frac12$, the trace operator on $\Gamma$ is continuous and surjective from $H^s(\Omega)$ to 
$H^{s-1/2}(\Gamma)$ and has a continuous right-inverse $\mathcal E_\Gamma: H^{s-1/2}(\Gamma) \to 
H^s(\Omega)$, see e.g.\ Theorem~4.7.1 in \cite{Triebel}. Given $s_0 > \frac12$, the right-inverse can be chosen to be independent of $s < s_0$. The normal derivative 
$\partial_\nu y = (\nu \cdot \nabla y)|_{\Gamma}$ is thus continuous from $H^s(\Omega)$  to
$H^{s-3/2}(\Gamma)$ for each $s > \frac32$.

We emphasize that for $(y,y_\Gamma)\in \L^2$ the components $y$ and $y_\Gamma$ need not be related.
However, we often identify $(y, y_\Gamma)\in \H^k$ for $k\geq 1$ with its first component $y$ only and 
write $y\in \H^k$, with a slight abuse of notation. By the properties of the trace, the space $\H^k$ 
is a proper closed subspace of $H^k(\Omega)\times H^k(\Gamma)$. 

Finally, for open sets $\omega\subset \Omega$, we consider $L^2(\omega)$ as a closed subspace of 
$L^2(\Omega)$ by extending functions on $\omega$ by zero to $\Omega$.

\subsection{The Laplace-Beltrami operator.} \label{sec:beltrami} We refer to Chapter~3 of \cite{Jost}
or Sections~2.4 and 5.1 of \cite{Tay} for more details and proofs. The operator $\Delta_\Gamma$ on $\Gamma$
 is  given by 
$$(\Delta_\Gamma y)\circ \text{g} =  \frac{1}{\sqrt{\det\text{G}}} \sum_{i,j=1}^{N-1}\partial_i 
    (\sqrt{\det\text{G}}\, \text{g}^{ij}\partial_j (y\circ \text{g}))$$
in local coordinates $\text{g}$, where $\text{G} = (\text{g}_{ij})$ is the metric tensor corresponding 
to $\text{g}$ and $\text{G}^{-1} = (\text{g}^{ij})$ denotes its inverse. However, in this paper we will 
not use this local formula, but rather the surface divergence theorem
\begin{equation}\label{eq:sdt}
 \int_{\Gamma} \Delta_\Gamma y\,z \: dS 
= - \int_{\Gamma } \langle\nabla_\Gamma y, \nabla_\Gamma z\rangle_\Gamma \,dS, \qquad 
      y \in H^2(\Gamma), \;z\in H^1(\Gamma),
\end{equation}
where $\nabla_\Gamma$ is the surface gradient and $\langle\cdot, \cdot\rangle_\Gamma$ is the Riemannian 
inner product of tangential vectors on $\Gamma$. The Laplace-Beltrami operator with domain $H^2(\Gamma)$ 
is self-adjoint and negative  on $L^2(\Gamma)$, cf.\ p.\ 309 of \cite{Tay}, and it thus generates an 
analytic $C_0$-semigroup on $L^2(\Gamma)$. Hence, $\|y\|_{L^2(\Gamma)}+\|\Delta_\Gamma y\|_{L^2(\Gamma)}$ 
defines an equivalent norm on $H^2(\Gamma)$. Moreover, $\|y\|_{L^2(\Gamma)}
 +\|\nabla_\Gamma y\|_{L^2(\Gamma)}$ gives an equivalent norm on $H^1(\Gamma)$.

\subsection{The Laplacian with surface diffusion boundary conditions} 
On $\L^2$ we consider the linear operator
$$A =\left ( \begin{array}{cc} d\Delta & 0 \\ -d\partial_\nu & \delta \Delta_\Gamma
\end{array}\right), \qquad D(A) = \H^2.
$$
Observe that $\E_1 = H^1(0,T;\L^2)\cap L^2(0,T; D(A))$. Strong solutions of \eqref{eq:4a}--\eqref{eq:6a}
will belong to $\E_1$. We show that $A$ is selfadjoint and negative. Our well-posedness and regularity 
results for the underlying evolution  equations rely on this fact.
\begin{proposition}\label{prop:semigroup} 
The operator $A$ is densely defined, self-adjoint, negative  and generates an analytic 
$C_0$-semigroup $(e^{tA})_{t\ge0}$ on $\L^2$. We further have $(\L^2, \H^2)_{1/2,2} = \H^1.$
\end{proposition}
\begin{proof} \emph{Step 1.} It is straightforward to check that $\{(y,y|_{\Gamma})\,:\, y\in C^\infty(\overline{\Omega})\} \subset \H^2$ is dense in $\L^2$, hence $A$ is densely defined.

\emph{Step 2.} To prove the asserted properties of $A$, on $\L^2$ we introduce the densely defined, closed,
 symmetric, positive  sesquilinear form
$$\mathfrak a[y,z] = \int_\Omega d\,\nabla y\cdot \nabla\overline{z}\,dx 
            +  \int_\Gamma \delta\,\langle \nabla_\Gamma y,\nabla_\Gamma\overline{ z}\rangle_\Gamma\,dS, 
             \qquad D(\mathfrak a) = \H^1.$$
It induces a positive self-adjoint sectorial operator $\tilde A$ on $\L^2$ which is given as follows, see Theorems~VI.2.1 and VI.2.6 in \cite{Kato}. A function  $y\in \H^1$ belongs to $D(\tilde A)$  if and only if there is  $f\in \L^2$ such that
$\mathfrak a[y,\varphi] = \langle f,\varphi\rangle_{\L^2}$ 
for all $\varphi\in \H^1$, and in this case $\tilde A y = f$.  Integration by parts and \eqref{eq:sdt} yield that $\H^2\subset D(\tilde A)$ and 
$\tilde A y = -A y$  for all $y\in \H^2$,  i.e., $\tilde A$ is an extension of $-A$. 

\emph{Step 3.} To obtain $-A = \tilde A$ and that $A$ generates an analytic $C_0$-semigroup, it remains
to show that $\lambda-A$ is surjective for some  $\lambda> 0$. Let $\lambda\geq 1$. 
Recall that the Dirichlet Laplacian $\Delta_D$ with domain $H^2(\Omega)\cap H^1_0(\Omega)$
is negative and selfadjoint in  $L^2(\Omega)$. The bounded linear operator 
$S_D: L^2(\Omega)\times H^{3/2}(\Gamma)\to H^2(\Omega)$ given by 
\begin{equation}\label{eq:SD}
S_D(f,h)=(\lambda-d\Delta_D)^{-1}(f-(\lambda-d\Delta)\mathcal E_\Gamma h) + \mathcal E_\Gamma h
\end{equation}
yields the unique solution $y=S_D(f,h)$ in $H^2(\Omega)$ of
$$\lambda y- d\Delta y = f \text{ \ in }\Omega, \qquad y|_{\Gamma} = h \text{ \ on }\Gamma$$
for $(f,h)\in  L^2(\Omega)\times H^{3/2}(\Gamma)$, where $\mathcal E_\Gamma$  is the right-inverse of the 
trace operator. Moreover, the bounded linear operator 
$S_\Gamma=(\lambda-\delta \Delta_\Gamma)^{-1}: L^2(\Gamma)\to H^2(\Gamma)$ provides the unique solution 
$y_\Gamma = S_\Gamma(g)$ in $H^2(\Gamma)$ of
$$\lambda y_\Gamma - \delta \Delta_\Gamma y_\Gamma = g\text{ \ on }\Gamma$$
for $g\in L^2(\Gamma)$. Let $(f,g)\in \L^2$ be given. We define the map $S$ by
$$S(y) = S_D(f, S_\Gamma(g-d\partial_\nu y))$$
for  $y\in H^{7/4}(\Omega)$, say, 
and show below that $S$ is a strict contraction on $H^{7/4}(\Omega)$ if $\lambda\ge 1$ is sufficiently 
large. The resulting fixed point   $y = S(y)$ in $H^{7/4}(\Omega)$ then belongs to $H^2(\Omega)$, 
its trace $y|_\Gamma$ is contained in $H^2(\Gamma)$, and they satisfy
\[
 \lambda y- d\Delta y = f \text{ \ in }\Omega, \qquad 
  \lambda y|_\Gamma - \delta\Delta_\Gamma y|_\Gamma + d\partial_\nu y   = g \text{ \ on }\Gamma.
\]
In particular, $y\in \H^2$ and $(\lambda-A)y=(f,g)$, as required for the surjectivity of $\lambda-A$.

Since $\Delta_\Gamma$ and $\Delta$ generate an analytic semigroup, the corresponding resolvent estimates 
and the interpolation inequality \eqref{eq:interpol} show that 
\begin{equation}\label{est:resolv}
\|S_\Gamma (y_\Gamma)\|_{H^s(\Gamma)}\leq C\lambda^{\frac{s}{2}-1}\,\|y_\Gamma\|_{L^2(\Gamma)}
\quad\text{ and }\quad
\|(\lambda-d\Delta_D)^{-1}y\|_{H^s(\Omega)}\leq C\lambda^{\frac{s}{2}-1}\,\|y\|_{L^2(\Omega)}
\end{equation}
for $s\in [0,2]$, where here and below $C$ does not depend on $\lambda$ or the respective functions. 
We also need an estimate for $S_D$ with precise dependence on $\lambda$. Formulas \eqref{eq:SD} 
and \eqref{est:resolv} imply
\begin{align*}
 \|S_D(0,h)\|_{H^{7/4}(\Omega)} 
& \leq C\lambda^{-\frac18} \,\|(\lambda-d\Delta)\mathcal E_\Gamma h\|_{L^2(\Omega)} 
    +\| \mathcal E_\Gamma h\|_{H^{7/4}(\Omega)}\\
 &\leq C\lambda^{-\frac18} \,(\lambda\|\mathcal E_\Gamma h\|_{L^2(\Omega)} 
    + \|\Delta \mathcal E_\Gamma h\|_{L^2(\Omega)}) + C\|h\|_{H^{5/4}(\Gamma)}\\
 & \leq  C \lambda^{\frac78}\, \|h\|_{H^{1/8}(\Gamma)}+ C\|h\|_{H^{3/2}(\Gamma)}.
\end{align*}
For $y,z\in H^{7/4}(\Omega)$, these inequalities yield
\begin{align*}
\|S(y)-S(z)\|_{H^{7/4}(\Omega)}&\,  = \|S_D(0, -d S_\Gamma(\partial_\nu (y-z)))\|_{H^{7/4}(\Omega)} \\
&\, \leq C \lambda^{\frac78}\, \|S_\Gamma(\partial_\nu (y-z))\|_{H^{1/8}(\Gamma)} 
   + C\|S_\Gamma(\partial_\nu (y-z))\|_{H^{3/2}(\Gamma)} \\
&\, \leq C\lambda^{-\frac{1}{16}}\, \|\partial_\nu(y-z)\|_{L^2(\Gamma)} 
    + C \lambda^{-\frac14} \,\|\partial_\nu(y-z)\|_{L^2(\Gamma)} \\
&\, \leq C \lambda^{-\frac{1}{16}}\, \|y-z\|_{H^{7/4}(\Omega)}.
\end{align*}
Therefore $S$ is a strict contraction for sufficiently large $\lambda\ge1$, and we conclude $A=-\tilde A$.

\emph{Step 4.} Theorem~VI.2.23 of \cite{Kato} then says that $D(A^{1/2}) = D(\mathfrak a)= \H^1$. 
Moreover, Theorem 4.36 of \cite{Lunardi-Int} yields $D(A^{1/2}) = (\L^2, D(A))_{1/2,2}$, so that
$\H^1 = (\L^2, \H^2)_{1/2,2}.$
\end{proof}

We note a few basic properties of the space $\E_1$.

\begin{proposition} \label{prop:E1} The following assertions are true.
\begin{itemize}
 \item[\text{(a)}] For each $\tau \in [0,T]$, the trace space of $\E_1$ at $t = \tau$ equals $\H^1$. 
 We further have
 $$\E_1\hookrightarrow C([0,T]; \H^1).$$
 \item[\text{(b)}] The space $\E_1$ embeds compactly into $L^2(0,T;\L^2)$.
\end{itemize}
\end{proposition}
\begin{proof} (a) Corollary~1.14 of \cite{Lunardi-Int} implies that the trace space of $\E_1$ 
equals  $(\L^2, \H^2)_{1/2,2}$ and that $\E_1\hookrightarrow C([0,T]; (\L^2, \H^2)_{1/2,2}).$
 Thus the first assertion is a consequence of Proposition \ref{prop:semigroup}. 

(b) Since $\H^2\hookrightarrow H^2(\Omega)\times H^2(\Gamma)$ is compact in $\L^2$, the compactness of 
$\E_1\embed L^2(0,T;\L^2)$ follows from  the Aubin-Lions lemma, see e.g.\ Corollary~4 in \cite{Si}.
\end{proof}

\subsection{Existence, uniqueness and regularity of solutions}  \label{sec:well}
We are concerned with the following classes of solutions of \eqref{eq:4a}--\eqref{eq:6a}.

\begin{definition} \label{def:solutions}
Let $f\in L^2(\Omega_T)$,  $g\in L^2(\Gamma_T)$ and $y_0\in \L^2$. 
\begin{itemize}
\item[\text{(a)}] A \emph{strong solution of \eqref{eq:4a}--\eqref{eq:6a}} is a function $y\in \E_1$  
fulfilling \eqref{eq:4a}--\eqref{eq:6a} in $L^2(0,T;\L^2)$.
\item[\text{(b)}] A  \emph{mild solution of \eqref{eq:4a}--\eqref{eq:6a}} is a function 
$y\in C([0,T]; \L^2)$ satisfying
\begin{equation}\label{eq:mild}
 y(t) = e^{tA}y_0 + \int_0^t e^{(t-\tau)A}[f(\tau)-a(\tau)y(\tau), g(\tau)-b(\tau)y(\tau)]\,d\tau
      \qquad \text{in $\L^2$ \ for  $t\in [0,T]$}.
\end{equation}
\item[\text{(c)}] A \emph{distributional solution of \eqref{eq:4a}--\eqref{eq:6a}} is a function 
$y\in L^2(0,T;\L^2)$ such that for all $\varphi \in \E_1$ with $\varphi(T,\cdot) = 0$ we have
\begin{align}
 \int_{\Omega_T} y(-\partial_t \varphi - & d\Delta \varphi +a\varphi)\,dx \,dt  
  + \int_{\Gamma_T} y(-\partial_t\varphi- \delta \Delta_\Gamma \varphi + d\partial_\nu \varphi 
        + b\varphi)\,dS \,dt  \notag\\
 & = \int_{\Omega_T} f\varphi \,dx \,dt + \int_{\Gamma_T} g\varphi \,dS \,dt 
   + \int_\Omega y_0 \varphi(0,\cdot)\,dx + \int_\Gamma y_0 \varphi(0,\cdot) \,d S.\label{eq:distri}
\end{align}
\item[\text{(d)}] We call $y\in L^2(0,T;\L^2)$ a \emph{distributional solution of 
\eqref{eq:4a}--\eqref{eq:6a} with vanishing end value} if $y$ satisfies \eqref{eq:distri} for all 
$\varphi\in \E_1$.
\end{itemize}
\end{definition}

Since our controllability results rely on an observability estimate for a dual problem, we also 
have to look at the adjoint backward evolution equation
\begin{alignat}{2}
-\partial_t\varphi-d\Delta \varphi +a(t,x)\varphi & = f(t,x)&\qquad\quad 
    &\text{in }\Omega_T,\label{eq:1} \\
-\partial_t \varphi_\Gamma -\delta\Delta_\Gamma \varphi_\Gamma +d\partial_\nu \varphi + b(t,x)\varphi 
  & = g(t,x) &  &\text{on }\Gamma_T, \label{eq:2}\\
\varphi(T,\cdot) & = \varphi_T &  &\text{in }\overline{\Omega},\label{eq:3}
\end{alignat}
for given $\ph_T$ in $\H^1$ or in $\L^2$,  $f\in L^2(\Omega_T)$ and $g\in L^2(\Gamma_T)$.
As in Definition~\ref{def:solutions}, a \emph{strong solution of \eqref{eq:1}--\eqref{eq:3}} is 
a function $\ph\in \E_1$  fulfilling \eqref{eq:1}--\eqref{eq:3} in $L^2(0,T;\L^2)$, and  
a \emph{mild solution of \eqref{eq:1}--\eqref{eq:3}} is a function  $\ph\in C([0,T]; \L^2)$ satisfying
\begin{equation}\label{eq:mild-back}
\ph(t)=e^{(T-t)A}\ph_T+\int_t^T e^{(\tau-t)A}[f(\tau)-a(\tau)\ph(\tau),g(\tau)-\b(\tau)\ph(\tau)]\,d\tau    
        \qquad \text{for \  $t\in [0,T]$}.
\end{equation}
Considering \eqref{eq:4a}--\eqref{eq:6a} with the data $\hat{a}(t,x)=a(T-t,x)$, $\hat{b}(t,x)=b(T-t,x)$,
$\hat f(t,x)=f(T-t,x)$ and $\hat g(t,x)=g(T-t,x)$, one can  pass from  statements about 
\eqref{eq:4a}--\eqref{eq:6a}  to those about \eqref{eq:1}--\eqref{eq:3}, and vice versa,
by means of the transformation $t'= T-t$.  Hence, the following results on strong and mild solutions have 
straightforward analogues for the adjoint problem which can easily be proved by this  transformation.
We omit the details, but establish in Proposition~\ref{prop:mild}(f) a `solution formula' for homogenous
backward system which is crucial for our  main Theorem~\ref{thm:control-linear}.

We start with strong solutions of \eqref{eq:4a}--\eqref{eq:6a}. Proposition \ref{prop:E1}
implies that such a solution can only exist if $y_0\in  \H^1$, i.e., the initial data on $\Omega$ and $\Gamma$ is related by the trace.

\begin{proposition} \label{prop:MR} 
Let $f\in L^2(\Omega_T)$, $g\in L^2(\Gamma_T)$ and $y_0 \in \H^1.$ Then there exists a unique strong 
solution $y\in \E_1$ of \eqref{eq:4a}--\eqref{eq:6a}, which is also a mild solution. Given $R> 0$, 
there is a constant $C= C(R) >0$ such that for all $a$ and $b$ with $\|a\|_{\infty}, \|b\|_\infty\leq R$ and 
all data we have
\begin{equation}\label{eq:esti-2}
 \|y\|_{\E_1} \leq C \, (\|y_0\|_{\H^1} + \|f\|_{L^2(\Omega_T)} + \|g\|_{L^2(\Gamma_T)}).
\end{equation}

\end{proposition}
\begin{proof} 
 We set $F=(f,g)$ and $B(t)(\psi,\psi_\Gamma)= (a(t,\cdot)\psi,b(t,\cdot)\psi_\Gamma)$. Since $A$ is 
 selfadjoint and negative and $B(\cdot)$ is uniformly bounded, 
Theorem~3.1 of \cite{PS} yields the asserted unique solution $y\in \E_1$ of \eqref{eq:4a}--\eqref{eq:6a}
and the estimate \eqref{eq:esti-2}. Using that $y\in\E_1$ solves \eqref{eq:4a}--\eqref{eq:6a}, we compute
\begin{equation}\label{eq:VK0}
y(t)-e^{tA}y_0 = \int_0^t \partial_s [e^{(t-s)A}y(s)]\,ds
               = \int_0^t e^{(t-s)A}[-B(s) y(s)+F(s)]\,ds
\end{equation}
so that $y$ is also a mild solution of \eqref{eq:4a}--\eqref{eq:6a}. To obtain the asserted
uniformity of the constant in \eqref{eq:esti-2}, we note that Theorem~3.1 of \cite{PS} gives
 bounded linear operators $S(t,s)$  on $\L^2$ depending strongly continuously on  
 $0\le s\le t\le T$ such that
\begin{equation}\label{eq:VK1}
y(t)= S(t,0)y_0 +\int_0^t S(t,s)F(s)\,ds, \qquad t\in [0,T].
\end{equation}
Taking $F=0$ and varying the initial time, \eqref{eq:VK0} yields
\begin{equation}\label{eq:VK2} 
S(t,\tau)y_0= e^{(t-\tau)A}y_0 - \int_\tau^t  e^{(t-s)A}B(s)S(s,\tau)y_0\,ds
\end{equation}
at first for $y_0\in\H^1$ and then for $y_0\in\L^2$ by approximation. From Gronwall's
inquality we now deduce that $\|S(t,s)\|\le C=C(R)$ for all $0\le s\le t\le T$. 
Due to \eqref{eq:VK1}, the strong solution of \eqref{eq:4a}--\eqref{eq:6a} thus satisfies
\begin{equation}\label{est:l2}
\|y(t)\|_{\L^2} \le  C\, (\|y_0\|_{\L^2} + \|f\|_{L^2(\Omega_T)} + \|g\|_{L^2(\Gamma_T)})
\end{equation}
for all $t\in[0,T]$. We further write $\tilde F =By+F$ so that $\partial_t y = Ay + \tilde F$. 
The estimates \eqref{eq:esti-2} for $a=b=0$ and  \eqref{est:l2} finally yield $C=C(R)$ with
\begin{align*}
\|y\|_{\E_1}&\le C\, (\|y_0\|_{\H^1} + \|\tilde F\|_{L^2(0,T;\mathbb L^2)})\\
&\le C\, (\|y_0\|_{\H^1} + \|f\|_{L^2(\Omega_T)} + \|g\|_{L^2(\Gamma_T)}
                  +R\, \|y\|_{L^2(0,T;\L^2)})\\
       &\le C \, (\|y_0\|_{\H^1} + \|f\|_{L^2(\Omega_T)} + \|g\|_{L^2(\Gamma_T)}). \qedhere
\end{align*}
\end{proof}

We next consider mild and distributional solutions for initial data in $\L^2$.

\begin{proposition} \label{prop:mild} Let $f\in L^2(\Omega_T)$, $g\in L^2(\Gamma_T)$ and 
$y_0,\ph_T\in \L^2$. Then the following holds true.
\begin{itemize}
 \item[\text{(a)}] There is a unique mild solution $y\in C([0,T];\L^2)$ of \eqref{eq:4a}--\eqref{eq:6a}. 
 The solution map $(y_0,f,g)\mapsto y$ is linear and continuous from $\L^2\times L^2(\Omega_T)\times 
 L^2(\Gamma_T)$ to $C([0,T]; \L^2)$. Moreover, $y$ belongs to $\E_1(\tau,T)$ and 
 solves \eqref{eq:4a}--\eqref{eq:6a} strongly on $(\tau,T)$  with initial $y(\tau)$, for each $\tau\in(0,T)$. 
 There are bounded linear operators $S(t,s)$ on $\L^2$ depending strongly continuous on $0\le s\le t\le T$ 
 such that
\begin{equation}\label{eq:VK}
y(t)= S(t,0)y_0 +\int_0^t S(t,s)(f(s),g(s))\,ds, \qquad t\in[0,T].
\end{equation}
\item[\text{(b)}] Given $R> 0$, there is a constant $C= C(R) >0$ such that for all $a$ and $b$ 
with $\|a\|_{\infty}, \|b\|_\infty\leq R$ and all data the mild solution of $y$ of \eqref{eq:4a}--\eqref{eq:6a} satisfies
\begin{equation}\label{eq:esti-1}
 \|y\|_{C([0,T];\L^2)} \leq C (\|y_0\|_{\L^2} + \|f\|_{L^2(\Omega_T)} + \|g\|_{L^2(\Gamma_T)}).
\end{equation}
 \item[\text{(c)}] If $y_0\in \H^1$, then the mild solution from (a) is the strong one.
 \item[\text{(d)}] A function $y$ is a distributional solution of \eqref{eq:4a}--\eqref{eq:6a} 
 if and only if it is a mild solution.
  \item[\text{(e)}] A distributional solution of $y$ \eqref{eq:4a}--\eqref{eq:6a} with vanishing end value 
  satisfies $y(T,\cdot) = 0$.
\item[\text{(f)}] The backward problem \eqref{eq:1}--\eqref{eq:3} with $f=g=0$  
has the unique mild solution given by $\ph(t)=S(T,t)^*\ph_T$, where $S(T,t)^*$ denotes the adjoint of $S(T,t)$. It is the unique strong solution 
if $\ph_T\in \H^1$.  
\end{itemize}
\end{proposition}
\begin{proof} 
(a)--(c). A straightforward Gronwall estimate shows uniqueness of mild solutions. Theorem~3.1 of \cite{PS} 
gives the desired operators $S(t,s)$ satisfying  \eqref{eq:VK}.  Above we have seen that these operators 
also fulfill the integral equation  \eqref{eq:VK2}. Inserting  \eqref{eq:VK2} into  \eqref{eq:VK} and 
interchanging integrals, we see that the function given by  \eqref{eq:VK} is a mild solution. The linearity 
and continuity of the solution map follows from \eqref{eq:VK} and the asserted regularity of $y$ is shown 
in Theorem~3.1 of \cite{PS}. Part (b) can be checked as in \eqref{est:l2} above, and (c) is a consequence
of Proposition~\ref{prop:MR} and the uniqueness in assertion (a).

(d)  Let $F\in L^2(0,T;\L^2)$ and $y_0\in\H^1$. The function $v$ defined by 
$v(t) = e^{tA}y_0+\int_0^t e^{(t-s)A}F(s)\,ds$ belongs to $\E_1$ and satisfies $\partial_t v = Av + F$ on 
$(0,T)$, due to Proposition~\ref{prop:MR} with $a=b=0$. Using the self-adjointness of $A$ and integration by 
parts in time, we thus obtain
\begin{align}
 \int_0^T \Big\langle e^{tA}y_0 + \int_0^t e^{(t-s)A} F(s)&\,ds, -\partial_t\varphi(t)  
     -A\varphi(t)\Big \rangle_{\L^2} \,d t\notag\\
 & = \int_0^T \langle F(t),\varphi(t)\rangle_{\L^2} \, dt 
              + \langle y_0, \varphi(0)\rangle_{\L^2}\label{m674}
\end{align} 
for all $\varphi\in \E_1$ with $\varphi(T) = 0$. By approximation, this identity also holds for $y_0\in \L^2$.
Now, let $y\in C([0,T];\L^2)$ be the mild solution of \eqref{eq:4a}--\eqref{eq:6a}. With the notation of the 
proof of Proposition~\ref{prop:MR} we have
\[ y(t) = e^{tA} y_0 + \int_0^t e^{(t-s)A}[(f(s),g(s))-B(s)y(s)]\,ds,\] 
and thus \eqref{m674} with $F = (f,g)$ implies that 
\begin{align*}
 \int_0^T \langle y(t), -\partial_t\varphi(t)-A\varphi(t)\rangle_{\L^2}dt 
     + \int_0^T&\, \Big \langle \int_0^t e^{(t-s)A} B(s)y(s)\,ds, - \partial_t\varphi(t)-A\varphi(t)\Big 
       \rangle_{\L^2}\,d t\\
 &= \int_0^T \langle F(t),\varphi(t)\rangle_{\L^2} \, dt + \langle y_0, \varphi(0)\rangle_{\L^2}
\end{align*}
for each $\varphi$ as above.
Another application of \eqref{m674} with $y_0 = 0$ and $F = By$ yields that here the second term on 
the left-hand side equals $\int_0^T \langle B y(t), \varphi(t)\rangle_{\L^2}\, dt 
= \int_0^T \langle y(t), B \varphi(t)\rangle_{\L^2}\, dt$. Hence, the mild solution $y$ is a 
distributional solution.

To prove that a distributional solution is already the mild one, we show uniqueness of distributional 
solutions. Let $y,z\in L^2(0,T;\L^2)$ be such solutions. We then obtain
$$\langle y-z, -\partial_t \varphi-A\varphi +B\varphi\rangle_{L^2(0,T;\L^2)} = 0$$
for all $\varphi\in \E_1$ with $\varphi(T) = 0$. The version of Proposition~\ref{prop:MR} 
for the backward problem \eqref{eq:1}--\eqref{eq:3} says that for every $\psi\in L^2(0,T;\L^2)$ there 
is a function $\varphi\in \E_1$ with $\varphi(T) = 0$ such that
 $\psi = -\varphi' - A\varphi +B\varphi$. Therefore $y = z$, and uniqueness follows.

(e) Let $y$ be a distributional solution with vanishing end value. Parts (a) and (d) yield 
$y\in \E_1(\tau,T)$  and $\partial_t y = Ay-By + (f,g)$ on $(\tau,T)$, for all $\tau\in (0,T)$. 
Let $\varphi \in \E_1$ be supported in $(0,T]$. Integrating by parts and using the self-adjointness 
of $A$ and $B(t)$, we then derive
\begin{align*}
\langle y(T),\varphi(T)\rangle_{\L^2} 
  &= \int_0^T(\langle\partial_t y(t),\varphi(t)\rangle_{\L^2}
           + \langle y(t), \partial_t\varphi(t))\rangle_{\L^2})\,d t \\
  &= \int_0^T\langle y(t), \partial_t\varphi(t) +A\varphi(t)-B(t)\varphi(t)\rangle_{\L^2}\,d t 
   + \int_0^T \langle (f(t), g(t)),\varphi(t)\rangle_{\L^2} \,d t = 0
\end{align*}
from \eqref{eq:distri}.
Since every $\psi \in \H^2$ can be represented as $\psi = \varphi(T)$ with a function $\varphi$ 
as above  and $\H^2$ is dense in $\L^2$, we conclude that $y(T) = 0$.

(f) As explained before Proposition~\ref{prop:MR}, we can derive from our previous results
that the backward problem \eqref{eq:1}--\eqref{eq:3} with $f=g=0$  and $\ph_T\in \L^2$ has a unique 
mild solution $\ph$ which is the unique strong solution if $\ph_T\in \H^1$. Equation (3.4) in 
the proof of Theorem~3.1 in \cite{PS} further gives the formula
\[S(T,t)y_0= e^{(T-t)A}y_0 - \int_t^T S(T,\tau)B(\tau)e^{(\tau-t)A}y_0\,d\tau\]
in $\L^2$ for $t\in[0,T]$. Taking adjoints and using the self-adjointness of $A$ and $B(\tau)$, we derive
\[S(T,t)^* \ph_T= e^{(T-t)A}\ph_T - \int_t^Te^{(\tau-t)A}B(\tau) S(T,\tau)^*\ph_T\,d\tau\]
for $\ph_T\in \L^2$. In view of \eqref{eq:mild-back}, $S(T,\cdot)^* \ph_T$ is the mild solution of 
\eqref{eq:1}--\eqref{eq:3} with $f=g=0$.
\end{proof}

\section{The Carleman estimate}\label{sec:carleman}
In this section we prove a Carleman estimate for the backward adjoint linear problem \eqref{eq:1}--\eqref{eq:3}, which is the key to null controllability in the linear and semilinear case. The weights appearing in the Carleman 
estimate are the same as in \cite{FCG} for the case of Dirichlet boundary conditions. They are based 
on the following auxiliary function $\eta^0$, see Lemma~1.2 of \cite{FCG} and the references given there.

\begin{lemma}\label{lem:eta} Given a nonempty open set $\omega' \Subset  \Omega$, there is a function 
$\eta^0\in C^2(\overline{\Omega})$ such that 
$$\eta^0 > 0 \;\;\emph{\text{ in }}\Omega, \qquad \eta^0 = 0 \;\;\emph{\text{ on }}\Gamma, 
\qquad |\nabla \eta^0| > 0 \;\;\emph{\text{ in }}\overline{\Omega\backslash \omega'}.$$
\end{lemma}

Since $|\nabla \eta^0|^2 = |\nabla_\Gamma \eta^0|^2 + |\partial_\nu \eta^0|^2$ on $\Gamma$, 
the function $\eta^0$ in the lemma satisfies
\begin{equation}\label{eq:eta}
\nabla_\Gamma \eta^0 = 0, \qquad|\nabla \eta^0| =|\partial_\nu \eta^0|, \qquad 
  \partial_\nu \eta^0 \leq -c <0 \qquad \text{on }\Gamma,
\end{equation}
for some constant $c>0$.

Given $\omega' \Subset \Omega$, we take $\lambda, m>1$ and $\eta^0$ with respect to $\omega'$ as in  Lemma \ref{lem:eta}. Following \cite{FCG}, we then define the weight functions $\alpha$ and $\xi$  by
\begin{align*}
\alpha(x,t) &= (t(T-t))^{-1} \big( e^{2\lambda m \|\eta^0\|_\infty} 
       - e^{\lambda(m\|\eta^0\|_\infty+  \eta^0(x))}\big),\\
\xi(x,t) &= (t(T-t))^{-1} e^{\lambda(m\|\eta^0\|_\infty+ \eta^0(x))},
\end{align*}
for $x\in \overline{\Omega}$ and $t\in [0,T]$. Note that $\alpha$ and $\xi$ are $C^2$ and strictly 
positive on $(0,T)\times\overline{\Omega}$ and blow up as $t\to 0$ and as $t\to T$. Moreover, the 
weights are constant on the boundary $\Gamma$ so that
\begin{equation}\label{eq:xi1}
\nabla_\Gamma \alpha = 0 \qquad \text{and} \qquad \nabla_\Gamma \xi= 0 \qquad \text{on }\Gamma.
\end{equation}

We state the Carleman estimate.
\begin{lemma}\label{lem:carleman} 
Let $T>0$, $\omega \Subset  \Omega$ be nonempty and open, $d,\delta>0$, $a\in L^\infty(Q_T)$ and 
$b\in L^\infty(\Gamma_T)$. Choose a nonempty open set 
$\omega' \Subset \omega$. Define $\eta^0$, $\alpha$ and $\xi$ as above with respect to $\omega'$. 
Then there are constants $C>0$ and $\lambda_1,s_1 \ge 1$ such that 
\begin{align*}
 s^{-1} \int_{\Omega_T} &e^{-2s\alpha} \xi^{-1} (|\partial_t\varphi|^2 + |\Delta \varphi|^2)\,dx\,dt 
  + s^{-1}\int_{\Gamma_T}e^{-2s\alpha}\xi^{-1}(|\partial_t \varphi|^2 
   + |\Delta_\Gamma \varphi|^2)\,dS\,dt \\
  & \qquad + s\lambda^2 \int_{\Omega_T} e^{-2s\alpha} \xi |\nabla \varphi|^2 \,dx\,dt + s\lambda \int_{\Gamma_T} e^{-2s\alpha} \xi |\nabla_\Gamma \varphi|^2 \,dS\,dt \\
 &  \qquad + s^3\lambda^4\int_{\Omega_T} e^{-2s\alpha} \xi^3 |\varphi|^2 \, dx\,dt +  s^3\lambda^3 \int_{\Gamma_T} e^{-2s\alpha} \xi^3 |\varphi|^2 \, dS\,dt \\
  & \qquad + s\lambda \int_{\Gamma_T} e^{-2s\alpha} \xi|\partial_\nu \varphi|^2 \,dS\,dt  \\
&\leq C s^3\lambda^4 \int_{\omega_T} e^{-2s\alpha} \xi^3 |\varphi|^2 \, dx\,dt 
 + C \int_{\Omega_T} e^{-2s\alpha} |\partial_t\varphi+ d\Delta \varphi - a\varphi|^2 \,dx \,dt  \\
  &\qquad + C \int_{\Gamma_T} e^{-2s\alpha} |\partial_t\varphi + \delta \Delta_\Gamma \varphi 
    - d \partial_\nu \varphi- b\varphi|^2 \,dS\,dt
\end{align*}
for all $\lambda \geq \lambda_1$, $s \geq s_1$ and $\varphi\in \E_1$. Given $R > 0$, the constant $C=C(R)$ can be chosen independently of all $a,b$ with $\|a\|_\infty,\|b\|_\infty\le R$.
\end{lemma}
\begin{proof} 
We follow the strategy of Lemma~1.3 of \cite{FCG}, but we have to treat various new boundary terms
arising from the dynamic boundary condition. Rescaling in time, we restrict to the case $d = 1$. It can be 
seen  by convolution with mollifiers in space and time that $C^\infty([0,T]\times \overline{\Omega})$ is 
dense in $\E_1$. Since all terms in the asserted inequality are continuous with respect to the $\E_1$-norm,  
it suffices to consider smooth functions $\varphi$. Throughout $C$ denotes a generic constant which 
does not depend on $\lambda$, $s$, $a$, $b$, and $\varphi$ subject to the assumptions of the lemma.

\emph{Step 1. Change of variables.} 
Let $\varphi\in C^\infty([0,T]\times \overline{\Omega})$,  $\lambda \geq \lambda_1\ge1$
and $s\ge s_1\ge 1$ be given. We define
$$\psi:= e^{-s\alpha} \varphi, \qquad  f = e^{-s\alpha} (\partial_t \varphi+ \Delta \varphi - a\varphi), 
\qquad g = e^{-s\alpha} (\partial_t\varphi + \delta \Delta_\Gamma \varphi -  \partial_\nu \ph - b\ph).$$
Observe that these functions vanish exponentially at $t = 0$ and $t=T$.
We determine the problem solved by $\psi$. We first expand the spatial derivaties of $\alpha$ by the 
chain rule to bring $\eta^0$ into play, but we do not expand $\partial_t \alpha$. We calculate
\begin{align}
\nabla \alpha & = - \nabla \xi = - \lambda \xi \nabla \eta^0,\label{eq:xi2}\\
\Delta \alpha &
     = - \lambda^2 \xi |\nabla \eta^0|^2  - \lambda \xi \Delta \eta^0,\notag \\
\partial_t \psi & = e^{-s\alpha} \partial_t \varphi - s \psi\partial_t\alpha, \notag\\
\nabla \psi & = e^{-s\alpha} \nabla \varphi - s \psi\nabla\alpha 
 = e^{-s\alpha} \nabla \varphi + s \lambda \psi \xi \nabla \eta^0, \label{eq:psi}\\
\Delta \psi &
   = e^{-s\alpha} \Delta \varphi + \nabla (e^{-s\alpha})\cdot  \nabla \varphi 
           - s  \psi \Delta \alpha -s\nabla \psi  \cdot \nabla\alpha \notag\\
&= e^{-s\alpha} \Delta \varphi - 2s \nabla \psi \cdot \nabla \alpha - s^2 \psi |\nabla \alpha|^2 
     - s  \psi \Delta \alpha. \notag 
\end{align}
On $\Omega_T$ this yields transformed evolution equations
\begin{align}
\partial_t \psi + \Delta \psi -a\psi& = f- s\psi\big( \partial_t\alpha + \Delta \alpha\big)  
    - 2s \nabla \psi \cdot\nabla \alpha - s^2 \psi |\nabla \alpha|^2\label{m1} \\
& = f- s\psi \partial_t\alpha + s  \lambda^2 \psi \xi |\nabla \eta^0|^2  +s \lambda \psi \xi \Delta \eta^0
  +  2s \lambda \xi \nabla \psi \cdot \nabla \eta^0 - s^2\lambda^2  \psi \xi^2 |\nabla \eta^0 |^2.\notag 
\end{align}
Similarly, using \eqref{eq:eta} and \eqref{eq:xi1}, on $\Gamma_T$ we obtain
\begin{align}
\partial_t \psi + \delta \Delta_\Gamma \psi - \partial_\nu \psi - b\psi 
  &= g - s\psi \partial_t\alpha  - s\lambda \psi \xi \partial_\nu \eta^0.\label{m2}
\end{align}
Extending the corresponding decomposition in \cite{FCG}, we rewrite the equations \eqref{m1} and 
\eqref{m2} as
\begin{equation}\label{m3}
 M_1\psi + M_2\psi = \tilde f \quad \text{in }\Omega_T, \qquad 
  N_1\psi + N_2\psi = \tilde g \quad \text{on }\Gamma_T,
\end{equation}
with the abbreviations
\begin{align*}
M_1\psi &= -2 s  \lambda^2 \psi \xi |\nabla \eta^0|^2  - 2s \lambda \xi \nabla \psi  \cdot \nabla \eta^0 
           + \partial_t\psi = (M_1\psi)_1 + (M_1\psi)_2 + (M_1\psi)_3 , \\ 
M_2\psi & = s^2\lambda^2  \psi \xi^2 |\nabla \eta^0 |^2 + \Delta \psi + s\psi \partial_t\alpha 
            = (M_2\psi)_1 + (M_2\psi)_2 + (M_2\psi)_3,\\ 
N_1\psi & =  \partial_t\psi+ s\lambda \psi \xi  \partial_\nu \eta^0 = (N_1\psi)_1 + (N_1\psi)_2,\\ 
N_2\psi & = \delta \Delta_\Gamma \psi + s\psi \partial_t\alpha - \partial_\nu \psi 
     = (N_2\psi)_1 + (N_2\psi)_2 + (N_2\psi)_3, \\
\tilde f & = f +s \lambda \psi \xi \Delta \eta^0 - s  \lambda^2 \psi \xi |\nabla \eta^0|^2 + a\psi,\\
\tilde g &= g + b\psi.
\end{align*}
Applying $\|\cdot\|_{L^2(\Omega_T)}^2$ resp.\ $\|\cdot\|_{L^2(\Gamma_T)}^2$ to the equations in \eqref{m3} 
and adding the resulting identities, we obtain
\begin{align}
\|\tilde f\|_{L^2(\Omega_T)}^2 + \|\tilde g\|_{L^2(\Gamma_T)}^2
 &= \|M_1\psi\|_{L^2(\Omega_T)}^2 + \|M_2\psi\|_{L^2(\Omega_T)}^2  +\|N_1\psi\|_{L^2(\Gamma_T)}^2
    + \|N_2\psi\|_{L^2(\Gamma_T)}^2\label{m12} \\
& \qquad+ 2 \sum_{i,j} \langle  (M_1\psi)_i, (M_2\psi)_j\rangle_{L^2(\Omega_T)} 
  + 2 \sum_{i,j} \langle (N_1\psi)_i, (N_2\psi)_j\rangle_{L^2( \Gamma_T)}.  \notag 
\end{align}

\emph{Step 2. Estimating the mixed terms in \eqref{m12} from below.} 
We often use  the following basic pointwise estimates on $\overline{\Omega}$, 
\begin{equation}
 e^{2\lambda m \|\eta^0\|_\infty} \leq e^{2\lambda (m \|\eta^0\|_\infty + \eta^0)},\qquad 
 |\nabla \alpha| \leq C \lambda \xi, \qquad |\partial_t \alpha| \leq C \xi^2, \qquad 
   |\partial_t \xi|\leq C\xi^2.\label{m6}
\end{equation}

\emph{Step 2a.}  We start with the negative term
$$\langle (M_1\psi)_1, (M_2\psi)_1\rangle_{L^2(\Omega_T)} 
       = - 2s^3\lambda^4 \int_{\Omega_T} |\nabla \eta^0|^4 \xi^3 \psi^2 \,d x\, d t.$$
Using integration by parts, \eqref{eq:xi2} and \eqref{eq:eta}, we further derive
\begin{align*}
 \langle (M_1\psi)_2&\, , (M_2\psi)_1\rangle_{ L^2(\Omega_T)} 
 = - s^3\lambda^3 \int_{\Omega_T} |\nabla \eta^0|^2 \xi^3  \nabla \eta^0 \cdot \nabla (\psi^2)\,d x\, d t\\
&= s^3\lambda^3 \int_{\Omega_T} \text{div}\big(|\nabla \eta^0|^2 \xi^3\nabla \eta^0\big)\psi^2\,dx\,dt  
- s^3\lambda^3 \int_{\Gamma_T} |\nabla \eta^0|^2 \xi^3  \partial_\nu \eta^0 \psi^2  \,d S\, d t\\
& = 3 s^3 \lambda^4 \int_{\Omega_T} |\nabla \eta^0|^4 \xi^3 \psi^2 \,dx \, dt
    + s^3 \lambda^3 \int_{\Omega_T} \Delta \eta^0 |\nabla \eta^0|^2 \xi^3 \psi^2 \,dx \, dt\\
& \qquad + s^3\lambda^3\int_{\Omega_T}(\nabla (|\nabla \eta^0|^2)\cdot \nabla \eta^0)\xi^3\psi^2 \,dx\,dt 
  - s^3\lambda^3 \int_{\Gamma_T} |\partial_\nu \eta^0|^2 \xi^3  \partial_\nu \eta^0 \psi^2  \,d S\, d t\\
& =: B_1 + B_2 + B_3 - B_4.
\end{align*}
For sufficiently large $\lambda_1$, the fact that $\nabla \eta^0 \neq 0$ on 
$\overline{\Omega\backslash \omega'}$ implies
\begin{align*}
\langle (M_1\psi)_1&, (M_2\psi)_1\rangle_{L^2(\Omega_T)}  + \langle (M_1\psi)_2 , 
          (M_2\psi)_1\rangle_{L^2(\Omega_T)}\\
& \geq C s^3\lambda^4 \int_{\Omega_T} \xi^3 \psi^2 \, d x\, d t 
    - C s^3 \lambda^4 \int_{(0,T)\times \omega'} \xi^3\psi^2 \, dx \, dt - B_4\\
& =:  A - B -  B_4.
\end{align*}
Integrating by parts in time, we continue with
\begin{align*}
 \langle (M_1\psi)_3, (M_2\psi)_1\rangle_{L^2(\Omega_T)} 
  &= \frac{1}{2} s^2\lambda^2 \int_{\Omega_T} |\nabla\eta^0|^2 \xi^2 \partial_t (\psi^2) \, dx \, dt
   = - s^2\lambda^2 \int_{\Omega_T} |\nabla\eta^0|^2  \partial_t \xi \xi \psi^2 \, dx \, dt\\
& \geq - C s^2\lambda^2 \int_{\Omega_T}    \xi^3 \psi^2 \, dx \, dt,
\end{align*}
where we  employed \eqref{m6} and that $\psi(0)=\psi(T)=0$. This term is absorbed by $A$ for large 
$\lambda_1$. Altogether, we have  shown 
\begin{align*}
\sum_{i=1,2,3}\langle (M_1\psi)_i, (M_2\psi)_1\rangle_{L^2(\Omega_T)} & \geq   A -  B -B_4
  \ge A -  B + Cs^3\lambda^3 \int_{\Gamma_T} \xi^3 \psi^2  \,d S\, dt,
 \end{align*}
using also \eqref{eq:eta}. We remark that $A$ and $B_4$ are the dominant positive terms involving 
$\psi^2$, whereas $B$ will lead to a `control term' on the right hand side of the final estimate.

\emph{Step 2b.} Integration by parts and \eqref{eq:xi1} yield
\begin{align*}
 \langle (M_1\psi)_1&, (M_2\psi)_2\rangle_{L^2(\Omega_T)}  
    = - 2s\lambda^2 \int_{\Omega_T} |\nabla \eta^0|^2 \xi \psi \Delta \psi  \, dx\, dt\\
& = 2s\lambda^2 \int_{\Omega_T} \nabla (|\nabla \eta^0|^2 \xi \psi)\cdot  \nabla \psi  \, dx\, dt 
  -  2s\lambda^2 \int_{\Gamma_T} |\nabla \eta^0|^2 \xi \psi \partial_\nu \psi \, dS\, dt\\
& = 2s\lambda^2 \int_{\Omega_T} |\nabla \eta^0|^2 \xi   |\nabla \psi|^2  \, dx\, dt  
    + 2s\lambda^2 \int_{\Omega_T} \xi \psi \nabla (|\nabla \eta^0|^2)\cdot  \nabla \psi  \, dx\, dt \\
&\qquad +2s\lambda^3 \int_{\Omega_T}|\nabla \eta^0|^2 \psi  \xi \nabla \eta^0 \cdot\nabla\psi  \,dx\,dt 
 -  2s\lambda^2 \int_{\Gamma_T} |\nabla \eta^0|^2 \xi \psi \partial_\nu \psi \, dS\, dt\\
& =: C_1 + C_2 + C_3 - C_4.
\end{align*}
Similar as $B_1$ above, the summand $C_1$ will lead to a term controlling $|\nabla\psi|^2$.
We now apply  Young's inequality to 
 $(\lambda^2 \xi^{1/2}\psi) \big( \xi^{1/2}\nabla (|\nabla \eta^0|^2) \cdot \nabla \psi\big)$ and 
 $\big (s\lambda \psi  \xi  \big) \big( \nabla \eta^0  \cdot \nabla \psi\big)$, respectively, and estimate
\begin{align*}
|C_2| &\leq C s \lambda^4 \int_{\Omega_T} \xi \psi^2  \, dx\, dt 
  + C s \int_{\Omega_T} \xi |\nabla \psi|^2  \, dx\, dt,\\
|C_3| &\leq C s^2\lambda^4 \int_{\Omega_T} \xi^2 \psi^2 \, dx\, dt 
     + C \lambda^2 \int_{\Omega_T} |\nabla \psi|^2 \, dx\, dt.
\end{align*}     
It follows
\begin{align}
\langle (M_1\psi)_1, (M_2\psi)_2\rangle_{L^2(\Omega_T)} 
   & \geq  C_1 - C s^2\lambda^4 \!\int_{\Omega_T} \!\!\! \xi^2\psi^2\, dx \, dt 
    - C\int\int_{\Omega_T} \!\!\!(s\xi +\lambda^2) |\nabla\psi|^2\, dx \, dt-C_4.\label{m8}
\end{align}
The next summand is given by 
\begin{align*}
\langle (M_1\psi)_2&, (M_2\psi)_2\rangle_{L^2(\Omega_T)} 
      = -2s\lambda \int_{\Omega_T} (\nabla \eta^0 \cdot \nabla\psi) \xi  \Delta \psi \, dx \, d t \\
& = -2s\lambda \int_{\Gamma_T}  \nabla \eta^0 \cdot \nabla \psi   \partial_\nu \psi\xi  \, dS \, d t 
  + 2s\lambda\int_{\Omega_T}\nabla\big((\nabla \eta^0 \cdot \nabla\psi)\xi \big) \cdot \nabla \psi\,dx\,dt\\
& = -2s\lambda \int_{\Gamma_T}  \nabla \eta^0 \cdot \nabla \psi   \partial_\nu \psi\xi \, dS \, d t 
    + 2s\lambda \int_{\Omega_T} \xi\, (\nabla^2 \eta^0 \,\nabla\psi) \cdot \nabla \psi\, dx \, d t \\
& \qquad    + 2s\lambda^2 \int_{\Omega_T}  \xi |\nabla \eta^0 \cdot \nabla\psi|^2 \, dx \, d t 
   + s\lambda \int_{\Omega_T}  \nabla \eta^0 \cdot \nabla|\nabla\psi|^2\,\xi  \, dx \, d t\\
& =: D_1 + D_2 + D_3 + D_4,
\end{align*}
due to integration by parts and \eqref{eq:xi2}. The term $D_3$ is positive, and $D_2$ can be bounded by
$$D_2 \geq - C s\lambda \int_{\Omega_T} \xi |\nabla \psi|^2 \, dx \, dt.$$
For $D_4$, integration by parts and  \eqref{eq:xi2} yield
\begin{align*}
D_4 & \, =   s\lambda \int_{\Gamma_T}  \partial_\nu  \eta^0 |\nabla \psi|^2\,\xi  \, dS \, d t 
    - s\lambda \int_{\Omega_T}  \text{div}(\xi \nabla \eta^0) |\nabla\psi|^2  \, dx \, d t\\
& = s\lambda \int_{\Gamma_T}  \partial_\nu  \eta^0 |\nabla \psi|^2\,\xi  \, dS \, d t 
  - s\lambda^2 \int_{\Omega_T}   \xi   |\nabla \eta^0|^2 |\nabla\psi|^2  \, dx \, d t  
  - s\lambda \int_{\Omega_T}   \xi \cdot \Delta  \eta^0 |\nabla\psi|^2  \, dx \, d t  \\
&=: D_{41} + D_{42} + D_{43}.
\end{align*}
Since $D_{43}$ can be treated as $D_2$, we infer
\begin{align*}
 \langle (M_1\psi)_2, (M_2\psi)_2\rangle_{L^2(\Omega_T)} 
  & \geq -2 s\lambda \int_{\Gamma_T}  \nabla \eta^0 \cdot \nabla \psi   \partial_\nu \psi\xi \, dS \, d t 
    -C s\lambda \int_{\Omega_T} \xi |\nabla \psi|^2 \, dx \, dt  \\
& \qquad + s\lambda \int_{\Gamma_T}  \partial_\nu  \eta^0 |\nabla \psi|^2\,\xi  \, dS \, d t 
  - s\lambda^2 \int_{\Omega_T}   \xi   |\nabla \eta^0|^2 |\nabla\psi|^2  \, dx \, d t.
\end{align*}
Observe that the last term can go with $C_1$ above. Because $\nabla \psi$ vanishes at $t = 0$ and $t = T$
in view of \eqref{eq:psi}, we obtain
\begin{align}
 \langle (M_1\psi)_3, (M_2\psi)_2\rangle_{L^2(\Omega_T)} 
 &= \int_{\Omega_T} \partial_t \psi \Delta \psi \, dx \,d t 
  = \int_{\Gamma_T} \partial_t \psi \partial_\nu \psi \, dS \,d t.\label{m13}
\end{align}
We summarize the estimates of this step and invoke again that  $\nabla \eta^0 \neq 0$ on 
$\overline{\Omega\backslash \omega'}$. Using also \eqref{eq:eta}, we arrive at
\begin{align*}
\sum_{i=1,2,3}\langle (M_1&\,\psi)_i, (M_2\psi)_2\rangle_{L^2(\Omega_T)} 
 \geq s\lambda^2 \int_{\Omega_T} |\nabla \eta^0|^2 \xi   |\nabla \psi|^2  \, dx\, dt 
   - C s^2\lambda^4 \int_{\Omega_T} \xi^2\psi^2\, dx \, dt \\
&\qquad  - C \int_{\Omega_T} (s\xi +\lambda^2) |\nabla\psi|^2\, dx \, dt  
   -  2s\lambda^2 \int_{\Gamma_T} |\nabla \eta^0|^2 \xi \psi \partial_\nu \psi \, dS\, dt\\
&\qquad   -2s\lambda \int_{\Gamma_T}  \nabla \eta^0 \cdot \nabla \psi\partial_\nu \psi\xi\,dS\,dt
      - C s\lambda \int_{\Omega_T} \xi |\nabla \psi|^2 \, dx \, dt\\
&\qquad + s\lambda \int_{\Gamma_T} \partial_\nu \eta^0 |\nabla \psi|^2 \xi \, d S\, dt 
  + \int_{\Gamma_T} \partial_t \psi \partial_\nu \psi \, dS \,d t\\
&\geq C s\lambda^2 \int_{\Omega_T}  \xi   |\nabla \psi|^2  \, dx\, dt 
  - C s\lambda^2 \int_{(0,T)\times \omega'}  \xi   |\nabla \psi|^2  \, dx\, dt 
    - C s^2\lambda^4 \int_{\Omega_T} \xi^2\psi^2\, dx \, dt \\
&\qquad    -  2s\lambda^2 \int_{\Gamma_T} (\partial_\nu \eta^0)^2 \xi \psi \partial_\nu \psi \, dS\, dt   
         -2s\lambda \int_{\Gamma_T}  \partial_\nu  \eta^0    (\partial_\nu \psi)^2\xi  \, dS \, d t \\
&\qquad + s\lambda \int_{\Gamma_T} \partial_\nu \eta^0 |\nabla \psi|^2 \xi \, d S\, dt 
  + \int_{\Gamma_T} \partial_t \psi \partial_\nu \psi \, dS \,d t,
\end{align*}
 increasing $\lambda_1$ and $s_1$ if necessary.

\emph{Step 2c.} Employing \eqref{m6}, we estimate
\begin{align*}
\langle (M_1\psi)_1&, (M_2\psi)_3\rangle_{L^2(\Omega_T)} 
 = - 2s^2\lambda^2 \int_{\Omega_T} |\nabla \eta^0|^2 \partial_t \alpha \xi \psi^2 \, dx\, dt\geq 
    - C s^2\lambda^2  \int_{\Omega_T}  \xi^3 \psi^2 \, dx\, dt.
\end{align*}
This term is absorbed by $A$. Integration by parts, \eqref{eq:xi2} and \eqref{m6} next imply
\begin{align*}
\langle (M_1\psi)_2&, (M_2\psi)_3\rangle_{L^2(\Omega_T)} 
    = - s^2 \lambda \int_{\Omega_T} \partial_t \alpha \xi \nabla\eta^0 \cdot \nabla (\psi^2) \, dx\, dt\\
&= - s^2 \lambda \int_{\Gamma_T} \partial_t \alpha \xi \partial_\nu\eta^0 \psi^2 \, dS\, dt 
   +  s^2 \lambda \int_{\Omega_T} \text{div}(\partial_t \alpha \xi \nabla\eta^0) \psi^2 \, dx\, dt\\
& = - s^2 \lambda \int_{\Gamma_T} \partial_t \alpha \xi \partial_\nu\eta^0 \psi^2 \, dS\, dt 
  +  s^2 \lambda \int_{\Omega_T} \nabla (\partial_t \alpha) \cdot \nabla\eta^0 \xi \psi^2 \, dx\, dt\\
&\qquad +  s^2 \lambda \int_{\Omega_T} \partial_t \alpha \nabla \xi \cdot \nabla\eta^0 \psi^2 \, dx\, dt 
   +  s^2 \lambda \int_{\Omega_T} \partial_t \alpha \xi \Delta \eta^0 \psi^2 \, dx\, dt\\
&\geq - C s^2 \lambda  \int_{\Gamma_T} \xi^3 \psi^2 \, dS\, dt 
     - C s^2 \lambda^2  \int_{\Omega_T} \xi^3 \psi^2 \, dx\, dt.
\end{align*}
Integrating by parts with respect to time, we can derive
\begin{align}\label{eq:M33}
\langle (M_1\psi)_3, (M_2\psi)_3\rangle_{L^2(\Omega_T)} 
  &= \frac{s}{2} \int_{\Omega_T} \partial_t \alpha \partial_t (\psi^2)\, d x \,d t 
   =  - \frac{s}{2} \int_{\Omega_T} \partial_t^2 \alpha \psi^2\, d x \,d t\\
&\geq  -Cs \int_{\Omega_T} \xi^3 \psi^2\, d x \,d t, \notag
\end{align}
since $\psi$ vanishes at the endpoints and $|\partial_t^2 \alpha| \leq C \xi^3$. Choosing sufficiently 
large $s_1$ and $\lambda_1$, we conclude from the above inequalities that
\begin{align*}
 \sum_{i=1,2,3}\langle (M_1&\,\psi)_i, (M_2\psi)_3\rangle_{L^2(\Omega_T)} 
  \geq   - C s^2 \lambda^2 \int_{\Omega_T} \xi^3 \psi^2 \, dx\, dt 
    - C s^2 \lambda   \int_{\Gamma_T} \xi^3 \psi^2 \, dS\, dt.
\end{align*}

\emph{Step 2d.} We now consider the boundary terms $N_1$ and $N_2$, employing the surface divergence 
theorem \eqref{eq:sdt} several times. We first compute
$$\langle (N_1\psi)_1, (N_2\psi)_1\rangle_{L^2(\Gamma_T)} 
= \int_{\Gamma_T} \delta \partial_t \psi \Delta_\Gamma \psi \,dS\,dt 
= -\frac{\delta}{2} \int_{\Gamma_T}  \partial_t |\nabla_\Gamma \psi|^2 \,dS\,dt  = 0$$
by means of $\psi(0)=\psi(T)=0$. Moreover,  \eqref{eq:xi1} yields
\begin{align*}
 \langle (N_1\psi)_2&, (N_2\psi)_1\rangle_{L^2(\Gamma_T)} 
   = \delta s\lambda \int_{\Gamma_T} \partial_\nu \eta^0 \xi \Delta_\Gamma \psi \psi \, d S \, dt
    = - \delta s\lambda \int_{\Gamma_T} \langle \nabla_\Gamma(\partial_\nu \eta^0 \xi \psi), 
       \nabla_\Gamma \psi\rangle_\Gamma \, d S \, dt\\
& = -\delta s\lambda \int_{\Gamma_T}  \xi \psi \langle \nabla_\Gamma\partial_\nu \eta^0 , 
    \nabla_\Gamma \psi\rangle_\Gamma \, d S \, dt 
    - \delta s\lambda \int_{\Gamma_T} \partial_\nu \eta^0 \xi  |\nabla_\Gamma \psi|^2 \, d S \, dt.
\end{align*}
The next two terms are estimated by
\begin{align*}
\langle (N_1\psi)_1, (N_2\psi)_2\rangle_{L^2(\Gamma_T)} 
 &= \frac{s}{2} \int_{\Gamma_T} \partial_t\alpha \partial_t(\psi^2)\,dS\,dt 
    \geq  -C s \int_{\Gamma_T} \xi^3 \psi^2\, d S \,d t,\\\
 \langle (N_1\psi)_2, (N_2\psi)_2\rangle_{L^2(\Gamma_T)} 
 &= s^2\lambda \int_{\Gamma_T} \partial_\nu \eta^0 \partial_t \alpha \xi \psi^2 \, dS\, dt \geq 
  -Cs^2 \lambda \int_{\Gamma_T} \xi^3 \psi^2 \, d S\, dt,
\end{align*}
where we proceed as in \eqref{eq:M33}  and use \eqref{m6}, respectively. Finally, the summand
\begin{align*}
 \langle (N_1\psi&\,)_1, (N_2\psi)_3\rangle_{L^2(\Gamma_T)} 
     = - \int_{\Gamma_T} \partial_t \psi \partial_\nu \psi\, dS\, dt
\end{align*}
cancels with the one from \eqref{m13}, and
\begin{align*}
 \langle (N_1\psi)_2, (N_2\psi)_3\rangle_{L^2(\Gamma_T)} 
     = - s\lambda \int_{\Gamma_T}  \xi \partial_\nu \eta^0 \partial_\nu \psi \psi\, d S\, dt.
\end{align*}

\emph{Step 3. The transformed estimate.} We collect the final inequalities in Steps 2a--2d. 
Increasing $\lambda_1$ and $s_1$ if needed to absorb lower order terms, we arrive at
\begin{align*}
  \sum_{i,j} &\langle (M_1\psi)_i, (M_2\psi)_j\rangle_{L^2(\Omega_T)}
    +  \sum_{i,j} \langle (N_1\psi)_i, (N_2\psi)_j\rangle_{L^2(\Gamma_T)}\\
& \geq C s^3\lambda^4 \int_{\Omega_T} \xi^3 \psi^2 \, d x\, d t 
    C s^3 \lambda^4 \int_{(0,T)\times \omega'} \xi^3\psi^2 \, dx \, dt\\
&\qquad + C s^3\lambda^3 \int_{\Gamma_T}  \xi^3   \psi^2  \,d S\, d t
  + C s\lambda^2 \int_{\Omega_T}  \xi   |\nabla \psi|^2  \, dx\, dt \\
 &\qquad - C s \lambda^2 \int_{(0,T)\times \omega'}  \xi   |\nabla \psi|^2  \, dx\, dt  
 -  2s\lambda^2 \int_{\Gamma_T} (\partial_\nu \eta^0)^2 \xi \psi \partial_\nu \psi \, dS\, dt \\
 &\qquad   -2s\lambda \int_{\Gamma_T}  \partial_\nu  \eta^0    (\partial_\nu \psi)^2\xi  \, dS \, d t 
 + s\lambda \int_{\Gamma_T} \partial_\nu \eta^0 |\nabla \psi|^2 \xi \, d S\, dt\\
& \qquad - \delta s\lambda \int_{\Gamma_T}  \xi \psi \langle \nabla_\Gamma\partial_\nu \eta^0, 
   \nabla_\Gamma \psi \rangle_{\Gamma}\, d S \, dt
      - \delta s\lambda \int_{\Gamma_T} \partial_\nu \eta^0 \xi  |\nabla_\Gamma \psi|^2 \, d S \, dt\\
&\, \qquad - s\lambda \int_{\Gamma_T}  \xi \partial_\nu \eta^0 \partial_\nu \psi \psi\, d S\, dt.
\end{align*}
We combine this estimate with \eqref{m12}. The expressions for $\tilde f$ and $\tilde g$ lead to additional 
lower order terms which can be absorbed to the left-hand side for large $\lambda_1$ and $s_1$. Using also 
$|\nabla \psi|^2 = |\nabla_\Gamma \psi|^2 + |\partial_\nu \psi|^2$ and \eqref{eq:eta}, we deduce
\begin{align}
\|M_1\psi&\|_{L^2(\Omega_T)}^2 + \|M_2\psi\|_{L^2(\Omega_T)}^2  +\|N_1\psi\|_{L^2(\Gamma_T)}^2 
   + \|N_2\psi\|_{L^2(\Gamma_T)}^2 
 + s^3\lambda^4 \int_{\Omega_T} \xi^3 \psi^2 \, d x\, d t \notag \\
 & \qquad +  s\lambda^2 \int_{\Omega_T}  \xi   |\nabla \psi|^2  \, dx\,dt 
  +   s^3\lambda^3 \int_{\Gamma_T} \xi^3  \psi^2  \,d S\, dt 
  + s\lambda \int_{\Gamma_T} \xi  |\nabla_\Gamma\psi|^2  \,d S\, dt \notag\\
&\leq C \int_{\Omega_T} e^{-2s\alpha} |\partial_t\varphi + \Delta \varphi - a \varphi|^2\, dx\,d t 
+ C\int_{\Gamma_T} e^{-2s\alpha} |\partial_t\varphi + \delta\Delta_\Gamma \varphi - \partial_\nu \varphi
        - b \varphi|^2\, dS\,d t \notag\\
& \qquad + C s^3 \lambda^4 \int_{(0,T)\times \omega'} \xi^3\psi^2 \, dx \, dt 
   + C s\lambda^2 \int_{(0,T)\times \omega'}  \xi   |\nabla \psi|^2  \, dx\, dt\notag \\
&\qquad    +  Cs\lambda^2 \int_{\Gamma_T} (\partial_\nu \eta^0)^2 \xi \psi \partial_\nu \psi \, dS\, dt   
+ C s\lambda \int_{\Gamma_T}  \partial_\nu  \eta^0    (\partial_\nu \psi)^2\xi  \, dS \, d t \notag\\
&\qquad + Cs\lambda\int_{\Gamma_T}\xi |\partial_\nu \eta^0| \, |\nabla_\Gamma \psi|^2\,dS\,dt  
 + Cs\lambda\int_{\Gamma_T}\xi\psi |\nabla_\Gamma \partial_\nu \eta^0|\,|\nabla_\Gamma \psi|\,dS\,dt\notag\\
& \qquad + C s\lambda \int_{\Gamma_T}  \xi \partial_\nu \eta^0 \partial_\nu \psi \psi\, d S\, dt.\label{m666}
\end{align}
We denote the five latter boundary integrals on the right-hand side of \eqref{m666} by $I_1,\ldots, I_5$.  Because of \eqref{eq:eta}, the second one can be bounded by
\begin{equation}\label{m667}
I_2 \leq -C s\lambda \int_{\Gamma_T}(\partial_\nu \psi)^2\xi  \, dS \, d t.
\end{equation}
 Young's inequality with $\varepsilon > 0$ allows to estimate $I_1$ by
\begin{align}\notag
I_1&\le Cs\int_{\Gamma_T}|\lambda^{3/2}\xi^{1/2}\psi|\,|\lambda^{1/2}\xi^{1/2}\partial_\nu \psi|\,dS\,dt\\
& \leq C_\varepsilon s \lambda^3 \int_{\Gamma_T} \xi^3 \psi^2 \, dS\, dt 
    + \varepsilon s\lambda \int_{\Gamma_T} \xi (\partial_\nu \psi)^2 \, dS\, dt.\label{est:I2}
\end{align}
Fixing a sufficiently small $\e>0$, we can absorb the second summand in \eqref{est:I2} by the 
right-hand side of \eqref{m667}. Increasing $s_1$ if necessary, we can then control  the first 
summand in \eqref{est:I2} by the left-hand side of \eqref{m666}. The fifth boundary integral $I_5$
is treated analogously.

For the integral $I_3$ in \eqref{m666}, recall from Section~\ref{sec:beltrami} that $\|\cdot\|_{L^2(\Gamma)} + \|\nabla_\Gamma \cdot\|_{L^2(\Gamma)}$ defines an equivalent norm on 
$H^1(\Gamma)$. The interpolation inequality \eqref{eq:interpol} thus yields
$$ \|\nabla_\Gamma \psi\|_{L^2(\Gamma)}^2 \leq C\|\psi\|_{H^2(\Gamma)} \|\psi\|_{L^2(\Gamma)}.$$
Moreover, $\|\cdot\|_{L^2(\Gamma)} + \|\Delta_\Gamma \cdot\|_{L^2(\Gamma)}$ defines an equivalent norm on 
$H^2(\Gamma)$. Also using that $\xi(t,\cdot)$ is constant on $\Gamma$, the integral $I_3$ can be bounded by
\begin{align}
I_3 & \leq Cs\lambda \int_0^T \xi \|\nabla_\Gamma \psi\|_{L^2(\Gamma)}^2\,dt\notag \\
&\leq C \int_0^T \big(s^{-1/2} \xi^{-1/2} \|\psi\|_{H^2(\Gamma)}\big) \big(s^{3/2} \lambda 
                   \xi^{3/2} \|\psi\|_{L^2(\Gamma)}\big)\,dt\notag \\
& \leq \varepsilon s^{-1} \int_{\Gamma_T} \xi^{-1} |\Delta_\Gamma \psi|^2 \,dS\,dt 
  + C_\e s^3 \lambda^2 \int_{\Gamma_T} \xi^3 \psi^2 \,dS\,dt \label{m668}
\end{align}
for $\e\in(0,1]$. In a similar way, we estimate $I_4$ by
\begin{align}
I_4 &\leq s\lambda \int_{\Gamma_T} \xi \|\nabla_\Gamma \psi\|_{L^2(\Gamma)}^2  \, dt 
     +  Cs \lambda \int_{\Gamma_T}  \xi \psi^2 \, d S \, dt \notag \\
& \leq  \varepsilon s^{-1} \int_{\Gamma_T} \xi^{-1} |\Delta_\Gamma \psi|^2 \,dS\,dt 
   + C_\e s^3 \lambda^2 \int_{\Gamma_T} \xi^3 \psi^2 \,dS\,dt.\label{m669}
\end{align}
The second summands in \eqref{m668} and \eqref{m669} can be absorbed by the left-hand side of \eqref{m666}
choosing a sufficiently large $\lambda_1$  depending on $\e\in(0,1]$. Altogether, we thus arrive at
\begin{align}
\|&M_1\psi\|_{L^2(\Omega_T)}^2 + \|M_2\psi\|_{L^2(\Omega_T)}^2  +\|N_1\psi\|_{L^2(\Gamma_T)}^2 
  + \|N_2\psi\|_{L^2(\Gamma_T)}^2 + s^3\lambda^4 \int_{\Omega_T} \xi^3 \psi^2 \, dx\,dt \notag\\
&\qquad +s\lambda^2 \!\int_{\Omega_T}\!  \xi   |\nabla \psi|^2  \, dx\,dt 
+  s^3\lambda^3\! \int_{\Gamma_T} \!\xi^3  \psi^2  \,d S\, dt 
 + s\lambda \!\int_{\Gamma_T}  \!  \xi(\partial_\nu \psi)^2  \, dS \, d t
  + s\lambda\! \int_{\Gamma_T}\! \xi  |\nabla_\Gamma\psi|^2  \,d S\, dt\notag\\
&\leq C\int_{\Omega_T} e^{-2s\alpha} |\partial_t\varphi + \Delta \varphi - a \varphi|^2\, dx\,d t 
   +C \int_{\Gamma_T} e^{-2s\alpha} |\partial_t\varphi + \delta\Delta_\Gamma \varphi -d \partial_\nu \varphi
       - b \varphi|^2\, dS\,d t \notag\\
&\qquad +Cs^3 \lambda^4 \int_{\omega'_T} \xi^3\psi^2 \, dx \, dt 
  + Cs\lambda^2 \int_{\omega'_T}  \xi   |\nabla \psi|^2  \, dx\, dt
   +\varepsilon s^{-1} \int_0^T \xi^{-1} |\Delta_\Gamma \psi|^2 \,dS\, dt.\label{m18}
\end{align}
To put the last summand in \eqref{m18} to the left, we observe that 
$\delta \Delta_\Gamma \psi = N_2\psi - s\psi \partial_t \alpha + \partial_\nu \psi$. Combined with 
\eqref{m6}, this identity yields
\begin{align} \label{eq:LB}
I:=s^{-1}\!\int_{\Gamma_T} \!\xi^{-1} |\Delta_\Gamma \psi|^2\,dS \, dt  
 \leq C\|N_2\psi\|_{L^2(\Gamma_T)}^2 + C s \! \int_{\Gamma_T}\! \xi^3 \psi^2 \,d S\, dt 
       + C\!\int_{\Gamma_T}\!\xi (\partial_\nu \psi)^2\, d S\, dt.
\end{align}
We can now fix a sufficiently small $\e\in(0,1]$ and choose sufficiently large $\lambda_1$ and $s_1$  
such that \eqref{m18} becomes
\begin{align}
\|&M_1\psi\|_{L^2(\Omega_T)}^2 + \|M_2\psi\|_{L^2(\Omega_T)}^2  +\|N_1\psi\|_{L^2(\Gamma_T)}^2 
 + \|N_2\psi\|_{L^2(\Gamma_T)}^2  + s^3\lambda^4 \int_{\Omega_T} \xi^3 \psi^2 \, d x\, d t \notag\\
 &\qquad  +s\lambda^2\!\int_{\Omega_T} \! \xi   |\nabla \psi|^2  \, dx\,dt 
  +  s^3\lambda^3 \! \int_{\Gamma_T} \xi^3 \! \psi^2  \,d S\, dt 
   + s\lambda \! \int_{\Gamma_T}\!\xi(\partial_\nu \psi)^2\,dS\,d t
    + s\lambda \!\int_{\Gamma_T}\! \xi  |\nabla_\Gamma\psi|^2  \,d S\, dt\notag\\
& \leq C \int_{\Omega_T} e^{-2s\alpha} |\partial_t\varphi + \Delta \varphi - a \varphi|^2\, dx\,d t 
 + C\int_{\Gamma_T} e^{-2s\alpha} |\partial_t\varphi + \delta\Delta_\Gamma \varphi - \partial_\nu \varphi 
    - b \varphi|^2\, dS\,d t\notag\\
& \qquad +Cs^3 \lambda^4 \int_{\omega'_T} \xi^3\psi^2 \, dx \, dt 
    + C s\lambda^2 \int_{\omega'_T}  \xi   |\nabla \psi|^2  \, dx\, dt. \label{est:car0}
\end{align}

{\it Step 4. Inverting the transformation.} The inequality \eqref{eq:LB} allows to replace in 
\eqref{est:car0} the summand $\|N_2\psi\|_{L^2(\Gamma_T)}^2$ by the term $I$ times a constant, where we 
increase $\lambda_1$ and $s_1$ if necessary to absorb the lower order terms in \eqref{eq:LB}. Similarly,
from $\partial_t\psi= N_1 \psi- s\lambda \psi \xi  \partial_\nu \eta^0$, we deduce that 
\[ I':=s^{-1} \int_{\Gamma_T} \xi^{-1} |\partial_t\psi|^2\,dS \, dt  
 \leq C\|N_1\psi\|_{L^2(\Gamma_T)}^2 + C s\lambda^2  \int_{\Gamma_T} \xi \psi^2 \,d S\, dt, \]
and hence also $I'$ can be put on the left hand side of \eqref{est:car0}.  In a similar way one handles
the corresponding terms on $\Omega_T$, see (1.58) and (1.59) in \cite{FCG}. We thus infer 
\begin{align}\label{est:car1}
   s^{-1}&\int_{\Omega_T} \xi^{-1} (|\partial_t\psi|^2 + |\Delta \psi|^2 )\,dx \, dt  
  +s^{-1}\int_{\Gamma_T} \xi^{-1} (|\partial_t\psi|^2 + |\Delta_\Gamma \psi|^2)\,dS \, dt  \\
 &\qquad  + s^3\lambda^4 \int_{\Omega_T} \xi^3 \psi^2 \, d x\, d t 
  +s\lambda^2 \int_{\Omega_T} \xi   |\nabla \psi|^2  \, dx\,dt \notag\\
 &\qquad  +  s^3\lambda^3 \int_{\Gamma_T} \xi^3  \psi^2  \,d S\, dt 
   + s\lambda \int_{\Gamma_T}\xi(\partial_\nu \psi)^2\,dS\,d t
    + s\lambda \int_{\Gamma_T}\xi  |\nabla_\Gamma\psi|^2  \,d S\, dt\notag\\
& \leq C \int_{\Omega_T} e^{-2s\alpha} |\partial_t\varphi + \Delta \varphi - a \varphi|^2\, dx\,d t 
 + C\int_{\Gamma_T} e^{-2s\alpha} |\partial_t\varphi + \delta\Delta_\Gamma \varphi - \partial_\nu \varphi 
    - b \varphi|^2\, dS\,d t\notag\notag\\
& \qquad +Cs^3 \lambda^4 \int_{\omega_T} \xi^3\psi^2 \, dx \, dt 
    + C s\lambda^2 \int_{\omega'_T}  \xi   |\nabla \psi|^2  \, dx\, dt, \notag
\end{align}
also using $\omega'\subset \omega$. As on p.1409 of \cite{FCG}, one can now 
absorb the gradient term on the right-hand side by the integral on $\omega_T$ and the left-hand side.

It remains to insert $\psi=e^{-s\alpha} \ph$ into \eqref{est:car1}. The terms involving derivatives of
 $\psi$ then lead to various lower order terms which can be controlled by the other terms in 
\eqref{est:car1}. For summands on $\Omega_T$ this is done in Step~3 of the proof of Lemma~1.3 in 
\cite{FCG}. Since the new terms on the boundary can be treated in the same way, we omit the details. 
One thus obtains the asserted Carleman estimate for the original function $\varphi$.
 \end{proof}

\begin{remark} \label{rem:delta=0} Up to inequality \eqref{m666}, the arguments in the above proof 
remain valid also for $\delta = 0$. However, in our proof the assumption $\delta > 0$ is essential to 
put the third and fourth boundary integral $I_3$ and $I_4$ on the right-hand side of \eqref{m666} to the left. 
\end{remark}

\section{Null controllability of linear and semilinear problems}\label{sec:null}
In this section we apply the Carleman estimate to show null controllability for \eqref{eq:intro} and its generalizations.
Throughout we fix $T> 0$, $\omega \Subset \Omega$, $d,\delta > 0$,  $a\in L^\infty(Q_T)$ and $b\in L^\infty(\Gamma_T)$. We assume that $s_1$ and $\lambda_1$ are sufficiently large to apply Lemma~\ref{lem:carleman} for this data.

\subsection{Linear problems} 
We show the null controllability of the linear inhomogeneous system
\begin{alignat}{2}
\partial_t y -d\Delta y +a(t,x)y
  &= v(t,x) \one_{\omega} +f(t,x)   &  \qquad \quad &\text{in }\Omega_T,\label{eq:4} \\
\partial_t y -\delta \Delta_\Gamma y + d\partial_\nu y + b(t,x) y  
  & = g(t,x) &  &\text{on }\Gamma_T, \label{eq:5}\\
y(0,\cdot) & = y_0 & &\text{in }\overline{\Omega}, \label{eq:6}
\end{alignat}
  by a standard duality argument. To this end, we first derive an observability estimate for the 
homogeneous backward system
\begin{alignat}{2}
-\partial_t \varphi-d\Delta \varphi + a(t,x) \varphi 
   & = 0&  \qquad \quad &\text{in }\Omega_T,\label{eq:15}\\
-\partial_t \varphi -\delta\Delta_\Gamma \varphi +d\partial_\nu \varphi + b(t,x) \varphi 
   & =0 &  &\text{on }\Gamma_T,\label{eq:16}\\
\varphi(T,\cdot) & = \varphi_T & & \text{in }\overline{\Omega},\label{eq:18}
\end{alignat}
from the Carleman estimate. This result also yields the final state observability of the system 
\eqref{eq:4}--\eqref{eq:6} with $f=g=v=0$.

\begin{proposition}\label{lem:observa} 
There is a constant $C >0$ such that for all $\varphi_T\in \L^2$ the mild solution 
$\varphi$ of the backward problem \eqref{eq:15}--\eqref{eq:18} satisfies
\begin{equation}
 \|\varphi(0,\cdot)\|_{\mathbb L^2}^2  \leq C \int_{\omega_T} |\varphi|^2 \,dx\,dt.\label{eq:observa}
\end{equation}
Moreover, for $ y_0\in \L^2$ the mild solution $y$ of the forward problem \eqref{eq:4}--\eqref{eq:6} with $f=g=v=0$
satisfies
\begin{equation}
 \|y(T,\cdot)\|_{\mathbb L^2}^2  \leq C \int_{\omega_T} |y|^2 \,dx\,dt.\label{eq:observa-y}
\end{equation}
Given $R > 0$, the constants $C=C(R)$ can be chosen independently of all $a,b$ with $\|a\|_\infty,\|b\|_\infty\le R$.
 \end{proposition}
\begin{proof} 
We argue as in Section~1.2 of \cite{FCG}. Recall from Subsection~\ref{sec:well} that the backward 
problem \eqref{eq:15}--\eqref{eq:18} possesses a unique mild solution $\ph\in C([0,T];\L^2)$ which 
depends continuously in $\L^2$ on $\ph_T$. By density we can thus restrict ourselves to final values
$\varphi_T\in \H^1$, so that $\varphi\in \E_1$ is a strong solution due to the backward version of 
Proposition~\ref{prop:mild}.  For fixed $\lambda = \lambda_1$ and $s= s_1$, the Carleman estimate from 
Lemma~\ref{lem:carleman} thus implies
\begin{align*}
\iint_{(T/4,3T/4)\times\Omega}|\ph|^2\,dx\,dt +  \iint_{(T/4, 3T/4)\times \Gamma} |\ph_\Gamma|^2\,dS\,dt 
    \leq  C \int_{\omega_T}  |\varphi|^2 \, dx\,dt.
\end{align*}
On the other hand, the backward version of \eqref{eq:esti-1} on the time interval $[0,t]$ yields
$\|\varphi(0,\cdot)\|_{\mathbb L^2} \leq C \|\varphi(t,\cdot)\|_{\mathbb L^2}$, for all $t\in (0,T)$.
Integrating this inequality over $(T/4, 3T/4)$, we infer
\[\|\varphi(0,\cdot)\|_{\mathbb L^2} \leq C \int_{T/4}^{3T/4} \|\ph(t)\|_{\L^2}\, dt,\]
from which \eqref{eq:observa} follows. The first assertion is a direct consequence the estimates in display. The second one then follows by
a simple transformation as explained in Subsection~\ref{sec:well}.
 \end{proof}

We now establish the  null controllability of the linear system, where we allow for inhomogeneities 
with exponential decay at $t=0$ and $t=T$. To this end, we introduce the weighted $L^2$-spaces
\begin{align*}
Z_{\Omega} &= \{f\in L^2(\Omega_T): e^{s\alpha} \xi^{-3/2}f\in L^2(\Omega_T)\},  \qquad 
\langle f_1,f_2\rangle_{Z_{\Omega}} = \int_{\Omega_T} f_1f_2 \,e^{2s\alpha} \xi^{-3}\,dx\,dt,\\
Z_{\Gamma} &= \{g\in L^2(\Gamma_T): e^{s\alpha} \xi^{-3/2}g\in L^2(\Gamma_T)\}, \qquad 
\langle g_1,g_2\rangle_{Z_{\Gamma}} = \int_{\Gamma_T} g_1g_2 \,e^{2s\alpha} \xi^{-3}\,dS\,dt.
\end{align*}
In Proposition~\ref{lem:min} we weaken the assumptions on $f$ and $g$, requiring decay only at $t= T$.

\begin{theorem} \label{thm:control-linear} 
For all data $y_0\in \L^2$, $f\in Z_\Omega$ and 
$g\in Z_\Gamma,$ there is a control $v\in L^2(\omega_T)$ such that the mild solution $y$ of 
\eqref{eq:4}--\eqref{eq:6} satisfies $y(T,\cdot) = 0.$
\end{theorem}

\begin{proof} 
By a duality argument, the assertion will follow  from the observability estimate of the previous
proposition. To this aim, we define the bounded linear operator
$\mathcal T: L^2( \omega_T)\to \mathbb L^2$ by 
\[ \mathcal Tv= \int_0^T S(T,\tau)(-\one_\omega v(\tau),0)\, d\tau,\]
where the solution operators $S(T,\tau)$ are given by Proposition~\ref{prop:mild}. Using that 
$Z_\Omega\times Z_\Gamma\hookrightarrow L^2(\Omega_T) \times L^2(\Gamma_T)$, we also introduce 
the bounded linear operator $\mathcal S: \mathbb L^2\times Z_\Omega\times Z_\Gamma \to \L^2$ by 
\[\mathcal S(y_0, f, g) = S(T,0)y_0+ \int_0^T S(T,\tau)(f(\tau),g(\tau))\,d\tau.  \]
Due to Proposition~\ref{prop:mild}(a), the function $\mathcal S (y_0,f,g) - \mathcal T v $ is the final 
value $y(T)$ of the mild solution $y$ to the inhomogeneous system \eqref{eq:4}--\eqref{eq:6}. Note that 
the adjoint $\mathcal T^*: \mathbb L^2 \to L^2(\omega_T)$ is given by 
\[ \mathcal T^* \varphi_T = -\one_\omega \varphi,\] 
where  $\varphi= S(T,\cdot)^*\ph_T$ is the mild solution  of the  homogeneous backward problem 
\eqref{eq:15}--\eqref{eq:18} with final value $\ph_T$, see Proposition~\ref{prop:mild}(f). 
 Taking the weights into account, one determines the adjoint $\mathcal S^*: \mathbb L^2 \to 
 \L^2\times Z_\Omega\times Z_\Gamma$ as
\[\mathcal S^* \varphi_T = (\varphi(0),e^{-2s\alpha} \xi^3 \varphi, e^{-2s\alpha} \xi^3 
                             \varphi|_{\Gamma}).\]
Recall from Subsection~\ref{sec:well} that $\ph\in\E_1$  is a strong solution of the backward
problem if $\ph_T\in\H^1$. The observability estimate from Proposition~\ref{lem:observa} 
 and the Carleman estimate from Lemma~\ref{lem:carleman} now imply that
\begin{align*}
\|\mathcal S^* \varphi_T\|_{\mathbb L^2\times Z_\Omega\times Z_\Gamma}^2 
& =\|\varphi(0)\|_{\L^2}^2 + \int_{\Omega_T} e^{-2s\alpha} \xi^3 \varphi^2 \,dx\,dt 
    + \int_{\Gamma_T} e^{-2s\alpha} \xi^3 \varphi^2 \,dS\,dt\\
& \leq C \int_{\omega_T} |\varphi|^2 \,dx\,dt = C \|\mathcal T^* \varphi_T\|_{L^2(\omega_T)},
\end{align*} 
at first for $\varphi_T\in \mathbb \H^1$, but then for $\varphi_T\in\L^2$ by approximation.  
Theorem~IV.2.2 of \cite{Zab} thus shows that the range of $\mathcal T$ contains that of $\mathcal S$; i.e.,
for all $y_0\in \mathbb L^2$, $f\in Z_\Omega$ and $g\in Z_\Gamma$ there is 
a control $v\in L^2(\omega_T)$ such that $\mathcal S (y_0,f,g) = \mathcal T v$. Therefore,  
$y(T) = \mathcal S (y_0,f,g) - \mathcal T v = 0$, as asserted.
 \end{proof}

The proof of null controllability in the semilinear case is based on a fixed point argument involving 
a continuous operator mapping the initial value to a null control of the inhomogeneous linear problem. 
To obtain such a operator, we have to single out a special control. As in Theorem~2.1 of \cite{I95} we choose 
a control having a minimal norm.

To that purpose, for given $a \in L^\infty(\Omega_T)$ and $b\in L^\infty(\Gamma_T)$ we define the 
backward parabolic operator 
$$L^* = \left ( \begin{array}{cc} -\partial_t - d \Delta + a  & 0 \\ d\partial_\nu & 
           -\partial_t - \delta \Delta_\Gamma +b   \end{array}\right)
      : \ \E_1 \to L^2(0,T;\L^2).$$
As in \cite{I95} we consider the weights 
$$\tilde \alpha(t,x) = \frac{t}{T}\,\alpha(t,x), \quad \tilde \xi(t,x) = \frac{t}{T}\,\xi (t,x),
\qquad t\in (0,T), \quad x\in \overline{\Omega},$$
which only blow up at $t=T$. We introduce  the weighted spaces
\begin{align*}
\tilde X &:= \big\{y\in L^2(0,T; \mathbb L^2)\,:\, e^{s \tilde \alpha} y \in L^2(0,T; \mathbb L^2)\big\},\\
\tilde{Z}_{\Omega} &:= \big\{f\in L^2(\Omega_T): e^{s\tilde{\alpha}} \tilde{\xi}^{-3/2}f 
    \in L^2(\Omega_T)\big\},\\
\tilde{Z}_{\Gamma}&:= \big\{g\in L^2(\Gamma_T): e^{s\tilde{\alpha}} \tilde{\xi}^{-3/2}g
 \in L^2(\Gamma_T)\big\}, 
 \end{align*}
endowed with the corresponding scalar products as above. Observe that the weights force decay only  
at $t = T$. We further define the functional $J$ by
$$J: \tilde X\times L^2(\omega_T)\to \R, \qquad J(y,v) 
    := \frac12 \int_{0}^T \|e^{s\tilde\alpha(t,\cdot)} y(t,\cdot)\|_{\mathbb L^2}^2\,dt 
   + \frac12 \int_{\omega_T } v^2 \,dx \,d t. $$

The general strategy of the following proof is the same as in \cite{I95}, even though we have not been able to 
obtain the Lagrange multiplier as in (2.5)--(2.6) of \cite{I95}. Hence our arguments differ in certain
important points from \cite{I95}, and we thus give all the details.

\begin{proposition}\label{lem:min} 
Let  $y_0 \in \mathbb \L^2$, $f\in \tilde{Z}_\Omega$ and 
$g\in \tilde{Z}_\Gamma$.  Then the functional $J$ has a unique minimizer 
  $(y,v)\in  \tilde X \times L^2(\omega_T)$ on the set
$$M =\big \{ (y,v)\in  \tilde X\times L^2(\omega_T)\,:\, (y,v) \text{ is a distributional
solution of  \eqref{eq:4}--\eqref{eq:6} with $y(T,\cdot) = 0$} \big\}.$$
There is a constant $C> 0$ such that for all 
 $y_0 \in \mathbb \L^2$, $f\in \tilde{Z}_\Omega$ and $g\in \tilde{Z}_\Gamma$ we have
\begin{equation}\label{est:XZ}
\|y\|_{\tilde X} + \|v\|_{L^2(\omega_T)} 
   \leq C\big( \|y_0\|_{\L^2} + \|f\|_{\tilde{Z}_\Omega} + \|g\|_{\tilde{Z}_\Gamma}\big).
\end{equation}
Given $R > 0$, the constant $C=C(R)$ can be chosen independently of all $a,b$ with $\|a\|_\infty,\|b\|_\infty\le R$.
\end{proposition}
\begin{proof} Since $M$ is a closed convex subset of $X\times L^2(\omega_T)$ and $J^{1/2}$ is an 
equivalent norm of $\tilde X\times L^2(\omega_T)$, the functional $J$ has a unique minimizer 
for  given  $y_0$, $f$ and $g$ if we can show that the set $M$ is nonempty. We will construct 
a function $(y_*,v_*)\in M$ as a weak limit of minimizers of regularized problems.

\emph{Step 1.} We first take $f\in \tilde{Z}_\Omega$ and $g\in \tilde{Z}_\Gamma$ supported in $(0,T]$. 
Hence, $f\in Z_\Omega$ and $g\in Z_\Gamma$. For each $\varepsilon \in(0,1]$, we introduce a new weight
$$\rho_\varepsilon(t,x) = e^{s  \tilde \alpha(t,x)\frac{T-t}{T-t+\varepsilon}},
    \qquad t\in (0,T),\quad x\in \overline{\Omega}.$$
which is bounded and strictly positive on $[0,T]\times \overline{\Omega}$. We define the corresponding
functional 
$J_\varepsilon: L^2(0,T; \L^2)\times L^2(\omega_T) \to \R$ by
$$J_\varepsilon(y,v) := \frac12 \int_{0}^T \|\rho_\varepsilon(t,\cdot) y(t,\cdot)\|_{\mathbb L^2}^2\,dt 
  + \frac12\int_{\omega_T } v^2 \,dx \,d t.$$
We further consider the map $G:  L^2(0,T;\L^2) \times L^2(\omega_T) \to \E_{1}^*$ given by
$$G(y,v)[\varphi] = \langle y, L^*\varphi\rangle_{L^2(0,T;\L^2)} - \langle y_0, \varphi(0)\rangle_{\L^2} 
        - \langle f, \varphi\rangle_{L^2(\Omega_T)} - \langle g,\varphi\rangle_{L^2(\Gamma_T)} 
         - \langle v,\varphi\rangle_{L^2(\omega_T)}$$
for $\varphi\in \E_{1}$. We point out that $(y,v)\in L^2(0,T;\L^2) \times L^2(\omega_T)$ is a 
distributional solution of \eqref{eq:4}--\eqref{eq:6} with $y(T,\cdot) = 0$ if and only if $G(y,v) = 0$. 
Observe that  $J_\varepsilon^{1/2}$ defines an equivalent norm on $L^2(0,T;\L^2) \times L^2(\omega_T)$
and that $\{(y,v)\in L^2(0,T;\L^2) \times L^2(\omega_T)\,:\,G(y,v) = 0\}$ is a closed convex subset of  
$L^2(0,T;\L^2) \times L^2(\omega_T)$. This subset is nonempty by Theorem~\ref{thm:control-linear}, since 
$f\in Z_\Omega$ and $g\in Z_\Gamma$. Hence, $J_\varepsilon$ has a unique minimizer 
$(y_\varepsilon,v_\varepsilon)$ on $\{G(y,v) = 0\}$ for every $\e \in(0,1]$.

\emph{Step 2.} We show that the functions $(\rho_\varepsilon y_\varepsilon,v_\varepsilon)$
are uniformly bounded in $L^2(0,T;\L^2)\times L^2(\omega_T)$ for $\e \in(0,1]$. To this end, we note that
each functional $J_\varepsilon$ is continuously differentiable and its derivative at $(y,v)$ is given by
$$J_\varepsilon'(y,v)[z,u] = \langle \rho_\varepsilon^2 y,z\rangle_{L^2(0,T;\L^2)} 
              + \langle v,u\rangle_{L^2(\omega_T)}, \qquad (z,u)\in L^2(0,T; \L^2)\times L^2(\omega_T).$$
The map $G$ is affine  and continuous. Its derivative is  the linear operator 
$G': L^2(0,T;\L^2) \times L^2(\omega_T) \to \E_1^*$ acting as  
$$G'(y,v)[\varphi]=  \langle y, L^*\varphi\rangle_{L^2(0,T;\L^2)} -\langle v,\varphi\rangle_{L^2(\omega_T)}
 \qquad \text{for \ } (y,v)\in  L^2(0,T;\L^2)\times L^2(\omega_T) \text{ \ and \ } \varphi\in \E_1.$$
For all $(z,u)\in\ker G'$ and $\sigma\in\R$ we have $J_\e(y_\e,v_\e)\le J_\e\big((y_\e,v_\e)+\sigma(z,u)\big)$,
since $(y_\e,v_\e)$ is the minimizer on $\{G(y,v) = 0\}$. At the minimum $\sigma=0$, we derive that 
$J_\varepsilon'(y_\varepsilon,v_\varepsilon)[z,u] = 0$ for all $(z,u)\in \ker G'$. Therefore
$$(\rho_\varepsilon^2 y_\varepsilon, v_\varepsilon) \in (\ker G')^\perp = \overline{\text{range}\, (G')^*}.$$
The adjoint $(G')^*:\E_1\to L^2(0,T;\L^2) \times L^2(\omega_T)$ is given by $(G')^*\varphi 
= (L^*\varphi, -\one_\omega\varphi).$ Thus there exists a sequence $(\varphi_n)_n$ in $\E_1$ such that
$$L^* \varphi_n \to \rho_\varepsilon^2 y_\varepsilon \text{ \ in }L^2(0,T;\L^2) \qquad \text{and} \qquad
          \qquad -\one_\omega \varphi_n \to v_\varepsilon\text{ \ in } L^2(\omega_T),$$
 as $n\to\infty$. Since $(y_\varepsilon, v_\varepsilon)$ is a distributional solution of 
\eqref{eq:4}--\eqref{eq:6} with $y_\e(T,\cdot) = 0$, we obtain
\begin{align}
\|\rho_\varepsilon y_\varepsilon&\|_{L^2(0,T;\L^2)}^2 + \|v_\varepsilon\|_{L^2(\omega_T)}^2
 = \langle y_\varepsilon, \rho_\varepsilon^2 y_\varepsilon\rangle_{L^2(0,T;\L^2)} 
        + \langle v_\varepsilon, v_\varepsilon\rangle_{L^2(\omega_T)}\notag\\
& = \lim_{n\to \infty} \big( \langle y_\varepsilon, L^*\varphi_n\rangle_{L^2(0,T;\L^2)} 
      - \langle v_\varepsilon, \varphi_n\rangle_{L^2(\omega_T)}\big)\notag\\
& =\lim_{n\to \infty}\big( \langle y_0, \varphi_n(0)\rangle_{\L^2} 
  + \langle f,\varphi_n\rangle_{L^2(\Omega_T)} + \langle g, \varphi_n\rangle_{L^2(\Gamma_T)}\big)\notag\\
&\leq  \limsup_{n\to \infty}\big(\|\varphi_n(0)\|_{\L^2}
                  +\|e^{-s\wt{\alpha}}\wt{\xi}^{3/2}\varphi_n\|_{L^2(0,T;\L^2)}\big)
  \big( \|y_0\|_{\L^2}  + \|f\|_{\wt{Z}_\Omega} + \|g\|_{\wt{Z}_\Gamma}\big).\label{m670}
\end{align}
The backward version of estimate \eqref{eq:esti-1} on $[0,t]$ yields
\begin{align}\label{est:phi-n}
\|\varphi_n(\tau)\|_{\L^2}^2 &\, \leq C\big( \|\ph_n(t)\|_{\L^2}^2 + \|L^*\ph_n\|_{L^2(\tau,t;\L^2)}^2\big)
\end{align}
for all $0\le \tau \le t\le T$, where $C$ is uniform in $a$ and $b$. We take $ \tau=0$ and integrate over 
$t\in (T/4,3T/4)$ in \eqref{est:phi-n}. The Carleman estimate then implies
\begin{align*}
\|\varphi_n(0)\|_{\L^2}^2 
   & \leq C\ \int_{T/4}^{3T/4} \|\varphi_n(t)\|_{\L^2}^2\,dt +  C \|L^*\varphi_n\|_{L^2(0,3T/4;\L^2)}^2\\
& \leq C \, \|e^{-s\alpha}\xi^{3/2} \varphi_n\|_{L^2(T/4,3T/4;\L^2)}^2 
         + C\, \| e^{-s\tilde\alpha} L^*\varphi_n\|_{L^2(0,3T/4; \L^2)}^2 \\
& \leq C \, \|e^{-s\alpha}\xi^{3/2} \varphi_n\|_{L^2(0,T;\L^2)}^2 
         + C\, \| e^{-s\tilde\alpha} L^*\varphi_n\|_{L^2(0,T; \L^2)}^2 \\
& \leq C \, \|e^{-s\alpha}\xi^{3/2} \varphi_n\|_{L^2(\omega_T)}^2
         +   C\, \| e^{-s\alpha} L^*\varphi_n\|_{L^2(0,T; \L^2)}^2 
         + C\, \| e^{-s\tilde\alpha} L^*\varphi_n\|_{L^2(0,T; \L^2)}^2 \\
& \leq C \int_{\omega_T}  |\varphi_n|^2\,dx\,dt + C\, \| e^{-s\tilde\alpha} L^*\varphi_n\|_{L^2(0,T;\L^2)}^2,\\
\limsup_{n\to\infty} \|\varphi_n(0)\|_{\L^2}^2 
            &\le C\,\|v_\e\|_{L^2(\omega_T)}^2 + C\,\| \rho_\e y_\e \|_{L^2(0,T; \L^2)}^2,
\end{align*}
 using also that $e^{-s\tilde \alpha}\rho_\e \leq 1$. If we let $\tau\in(0,T/4)$ and $t=\tau+T/4$ in 
\eqref{est:phi-n}, we further derive
\begin{align*}
\|e^{-s\wt{\alpha}}\wt{\xi}^{3/2} \varphi_n\|_{L^2(0,T/4;\L^2)}^2 
   &\le C\, \|\varphi_n\|_{L^2(0,T/4;\L^2)}^2 \\
   &\leq C\,\|\varphi_n\|_{L^2(T/4,T/2;\L^2)}^2+  C\, \|L^*\varphi_n\|_{L^2(0,T/2;\L^2)}^2\\
   &\le    C \,\|e^{-s\alpha}\xi^{3/2} \varphi_n\|_{L^2(T/4,T/2;\L^2)}^2+
          C\,\| e^{-s\tilde\alpha} L^*\varphi_n\|_{L^2(0,T/2; \L^2)}^2,\\
\|e^{-s\wt{\alpha}}\wt{\xi}^{3/2} \varphi_n\|_{L^2(0,T;\L^2)}^2 
& \le C \,\|e^{-s\alpha} \xi^{3/2} \varphi_n\|_{L^2(0,T;\L^2)}^2 
       + C\, \|e^{-s\tilde\alpha}L^*\varphi_n\|_{L^2(0,T;\L^2)}^2.
\end{align*}
As above, the Carleman estimate now yields
\begin{align*}
\limsup_{n\to\infty} \|e^{-s\wt{\alpha}}\tilde{\xi}^{3/2} \varphi_n\|_{L^2(0,T;\L^2)}^2 
& \leq C \limsup_{n\to\infty}(\|\varphi_n\|_{L^2(\omega_T)}^2 + 
            \|e^{-s\wt\alpha} L^*\varphi_n\|_{L^2(0,T;\L^2)}^2)\\
& \le C\,\|v_\e\|_{L^2(\omega_T)}^2 + C\,\| \rho_\e y_\e \|_{L^2(0,T; \L^2)}^2.
\end{align*}
Dividing \eqref{m670} by $\| \rho_\e y_n \|_{L^2(0,T; \L^2)} + \|v_\e\|_{L^2(\omega_T)}$, we arrive at
\begin{equation}\label{m672}
\|y_\e\|_{L^2(0,T;\L^2)} \le \|\rho_\e y_\e\|_{L^2(0,T;\L^2)} + \|v_\varepsilon\|_{L^2(\omega_T)}
   \leq C\big(\|y_0\|_{\L^2}+  \|f\|_{\tilde Z_\Omega} + \|g\|_{\tilde Z_\Gamma}\big),
\end{equation}
where the constant $C$ is uniform in $a$, $b$, and $\e$.

\emph{Step 3.} As a consequence of \eqref{m672} we find $(y_*,v_*)\in \tilde X\times L^2(\omega_T)$ such that,
 up to subsequences, $(y_\e, v_\e)$ tends to $(y_*,v_*)$ weakly in $L^2(0,T;\L^2)\times L^2(\omega_T)$ 
and $\rho_\e y_\e$ converge to  some $z$ weakly in $L^2(0,T;\L^2)$ as $\varepsilon\to 0$. 
Note that $\rho_\e$ tend pointwise to $e^{s\tilde \alpha}$ and $0\le \rho_\e\le e^{s\tilde \alpha}$.
If $\psi\in  L^2(0,T;\L^2)$ is supported in $(0,T)$, the functions $\rho_\e \psi$ thus converge  to  
$e^{s\tilde \alpha} \psi$  in  $L^2(0,T;\L^2)$ by dominated convergence. It follows that 
$z=e^{s\tilde \alpha}y_*$ and hence $y_*\in \tilde X$. The limit $(y_*,v_*)\in \tilde X\times L^2(\omega_T)$ 
is still a distributional solution of \eqref{eq:4}--\eqref{eq:6} with $y_*(T,\cdot) = 0$, i.e., 
$(y_*,v_*)\in M$. 

Therefore,  $J$ has a unique minimizer  $(y,v) \in M$ on the nonempty subset $M$. Using  the weak
convergence and \eqref{m672}, we derive the estimate \eqref{est:XZ} for $(y,v)$ by
\begin{align}
 \| y\|_{\tilde X} + \|v\|_{L^2(\omega_T)} & \le C J(y,v)^{1/2} \leq C J(y_*,v_*)^{1/2}
 \le C (\|e^{s\tilde \alpha} y_*\|_{L^2(0,T;\L^2)} + \|v_*\|_{L^2(\omega_T)})  \notag  \\
 & \leq  C\liminf_{\e\to 0} (\|\rho_\e y_\e\|_{L^2(0,T;\L^2)} + \|v_\e\|_{L^2(\omega_T)})\label{m673}\\
 &\,\leq  C\big(\|y_0\|_{\L^2}+  \|f\|_{\tilde Z_\Omega} + \|g\|_{\tilde Z_\Gamma}\big).\notag
\end{align}

\emph{Step 4.} Finally, let $f\in \tilde{Z}_\Omega$ and $g\in \tilde{Z}_\Gamma$ be the given inhomogeneities. 
Consider the solution set $M$ with respect to $f$ and $g$.  Choose $f_n\in \tilde{Z}_\Omega$ and 
$g_n\in \tilde{Z}_\Gamma$ with compact support in $(0,T]$ such that $f_n\to f$ in $\tilde{Z}_\Omega$ and 
$g_n\to g$ in $\tilde{Z}_\Gamma$ as $n\to \infty$. Let $(y_n,v_n)$ be the corresponding minimizers of $J$ 
obtained in Step 3 (where $M=M_n$ is defined for $f_n$ and $g_n$). Since these functions 
satisfy \eqref{est:XZ}, we find a subsequence such that $(y_{n_j},v_{n_j})$ 
tends to some $(y_*,v_*)$ weakly in $\tilde X\times L^2(\omega_T)$ as $j\to \infty$. The limit $(y_*,v_*)$ 
is a distributional solution of \eqref{eq:4}--\eqref{eq:6} with $y_*(T,\cdot) = 0$ so that $(y_*,v_*)\in M$. 
This implies as before that $J$ has a unique minimizer $(y,v)$ on $M$.  The estimate \eqref{est:XZ} for 
 $(y,v)$ can be shown as in \eqref{m673}.
\end{proof}

\subsection{Semilinear problems.}
With these preparations we can prove the null controllability of the inhomogeneous semilinear problem
\begin{alignat}{2}
\partial_t y -d\Delta y + F(y) & =v \one_{\omega} + f(t,x) & \qquad \quad & \text{in }\Omega_T,\label{eq:10}\\
\partial_t y -\delta \Delta_\Gamma y +d\partial_\nu y + G(y)  & = g(t,x)&  &\text{on }\Gamma_T, \label{eq:11}\\
y(0,\cdot)& = y_0  & &\text{in }\overline{\Omega}. \label{eq:13}
\end{alignat}

\begin{theorem} \label{thm:control-semilinear} Assume $F,G\in C^1(\R)$ satisfy
$$ F(0)=G(0) = 0 \quad\text{ and } \quad  |F(\xi)|+ |G(\xi)| \leq C(1+|\xi|) \quad \text{for \ } \xi\in \R.$$
Then for  all data $y_0\in \H^1$, $f\in \tilde Z_\Omega$ and $g\in \tilde Z_\Gamma$ there is a control 
$v\in L^2(\omega_T)$ such that \eqref{eq:10}--\eqref{eq:13} has a unique strong solution $y\in \E_1\cap \tilde X$ 
with $y(T,\cdot) = 0.$
\end{theorem}

\begin{proof} 
The result is proved by  Schauder's fixed point theorem. Based on Proposition~\ref{lem:min} we 
can follow the arguments given  in Theorem~3.1 of \cite{I95}. 

We write $F(\xi) = \tilde F(\xi) \xi$ and $G(\xi) = \tilde G(\xi) \xi$ for bounded functions $\tilde F$ and 
$\tilde G$. Fix data $y_0\in \H^1$, $f\in \tilde Z_\Omega$ and $g\in \tilde Z_\Gamma$. Take  
$y\in L^2(0,T;\L^2)$. Proposition~\ref{lem:min} then yields a unique  minimizer $(z_*,v_*) =: \Phi (y)$ of $J$
 among all distributional solutions in  $\tilde X\times L^2(\omega_T)$ of the linearized system
\begin{alignat}{2}
\partial_t z -d\Delta z +\tilde F(y(t,x)) z& =v \one_{\omega} + f(t,x) & \qquad\quad &\text{in }\Omega_T,\notag\\
\partial_t z -\delta \Delta_\Gamma z + d \partial_\nu z + \tilde G(y(t,x)) z  & = g(t,x) &  &\text{on }\Gamma_T,
  \label{m671} \\
z(T,\cdot) = 0, \quad z(0,\cdot)& = y_0 & &\text{in }\overline{\Omega}.\notag
\end{alignat}
We have thus  defined a map $\Phi: L^2(0,T;\L^2) \to \tilde X\times L^2(\omega_T).$ We write $z_*=\Phi_1(y)$ for
the first component of $\Phi(y)$. We consider $\Phi_1$ as an operator $\Phi_1: L^2(0,T;\L^2)\to L^2(0,T;\L^2)$. 

Since $\tilde F,\tilde G$ are bounded, the estimate \eqref{est:XZ} in Proposition~\ref{lem:min} shows that 
$\Phi_1$ maps all $y\in L^2(0,T;\L^2)$ into a ball in  $L^2(0,T;\L^2)$. Moreover, Proposition~\ref{prop:mild} 
says that $\Phi_1(y)$ is even a strong solution of \eqref{m671}. Proposition~\ref{prop:MR} and \eqref{est:XZ} 
thus yield
\begin{align*}
 \|\Phi_1(y)\|_{\E_1} 
    &\leq C\big(\|y_0\|_{\H^1}+\|v\|_{L^2(\omega_T)}+\|f\|_{L^2(\Omega_T)}+\|g\|_{L^2(\Gamma_T)}\big) \\
  &\leq C (\|y_0\|_{\H^1}+ \|f\|_{\tilde Z_\Omega} + \|g\|_{\tilde Z_\Gamma}),
\end{align*}
where $C$ does not depend on $y$. Since $\E_1$ is compactly embedded into $L^2(0,T;\L^2)$ by
 Proposition~\ref{prop:E1}, we conclude that $\Phi_1$ is compact.

 Employing Proposition~\ref{lem:min}, one can prove that  $\Phi_1$ is continuous as in  Theorem 3.1 of \cite{I95}
by a contradiction argument. Schauder's fixed point theorem thus gives a function $\hat y\in  L^2(0,T;\L^2)$ 
satisfying $\hat y=\Phi_1(\hat y)\in\E_1$. Hence, there is a control $\hat v\in L^2(\omega_T)$ such 
that  \eqref{m671} holds with $z=\hat y\in \tilde X$ and $v=\hat v$, as asserted.
 \end{proof}

\end{document}